\documentclass[preprint]{imsart}
\usepackage{amssymb, latexsym, amsmath}
\usepackage{amsfonts}
\usepackage{amsthm}
\usepackage{bbm}
\usepackage[T1]{fontenc}
\usepackage{courier}
\usepackage[includeall,
            vmargin={1in,1in},
            hmargin={1in,0.375in}
            ]{geometry}
\usepackage{tikz}
    
    \usetikzlibrary{patterns,decorations.pathreplacing}
\usepackage{url}
\usepackage{enumerate}

\allowdisplaybreaks

\theoremstyle{definition}
\newtheorem{theorem}{Theorem}[section]
\newtheorem{lemma}[theorem]{Lemma}

\newtheorem{corollary}[theorem]{Corollary}

\newtheorem{problem}[theorem]{Problem}

\def\EE{{\mathbb E}}
\def\PP{{\mathbb P}}
\def\NN{{\mathbb N}}
\def\ZZ{{\mathbb Z}}

\def\({\left(}
\def\){\right)}

\newcommand{\closure}[2]{\left\langle #1 \right\rangle_{#2}}
\newcommand{\indicator}[1]{\mathbbm{1}_{\{ #1 \}}}

\newenvironment{proofOfStatementOfTheorem}[2]{\begin{proof}[\textit{Proof of Statement #1 of Theorem \ref{#2}}]}{\end{proof}}

\begin{document}

\begin{frontmatter}

\title{The time of graph bootstrap percolation}

\begin{aug}
\author{\fnms{Karen} \snm{Gunderson}\thanksref{t2}\ead[label=e1]{karen.gunderson@umanitoba.ca}},
\author{\fnms{Sebastian} \snm{Koch}\thanksref{}\ead[label=e2]{sk629@cam.ac.uk}},
\and
\author{\fnms{Micha{\l}} \snm{Przykucki}\thanksref{t1}\ead[label=e3]{przykucki@maths.ox.ac.uk}}

\thankstext{t1}{Supported in part by EU project MULTIPLEX no. 317532.}
\thankstext{t2}{Research carried out while affiliated with the Heilbronn Institute for Mathematical Research at the University of Bristol.}

\affiliation{University of Manitoba; University of Cambridge; University of Oxford}

\address{K. Gunderson \\
Department of Mathematics \\
University of Manitoba \\
Winnipeg, MB R3T 2N2 \\
Canada \\
\printead{e1}}

\address{S. Koch \\
Department of Pure Mathematics and Mathematical Statistics \\
University of Cambridge \\
Cambridge CB3 0WB \\
United Kingdom \\
\printead{e2}}

\address{M. Przykucki \\
St Anne's College \\
University of Oxford \\ 
Oxford OX2 6HS \\
United Kingdom \\
\printead{e3}}

\runauthor{Gunderson, Koch, and Przykucki}
\runtitle{Time of graph bootstrap percolation}
\end{aug}

\begin{abstract}
Graph bootstrap percolation, introduced by Bollob{\'a}s in 1968, is a cellular automaton defined as follows. Given a ``small'' graph $H$ and a ``large'' graph $G = G_0 \subseteq K_n$, in consecutive steps we obtain $G_{t+1}$ from $G_t$ by adding to it all new edges $e$ such that $G_t \cup e$ contains a new copy of $H$. We say that $G$ percolates if for some $t \geq 0$, we have $G_t = K_n$.

For $H = K_r$, the question about the size of the smallest percolating graphs was independently answered by Alon, Frankl and Kalai in the 1980's. Recently, Balogh, Bollob{\'a}s and Morris considered graph bootstrap percolation for $G = G(n,p)$ and studied the critical probability $p_c(n,K_r)$, for the event that the graph percolates with high probability. In this paper, using the same setup, we determine, up to a logarithmic factor, the critical probability for percolation by time $t$ for all $1 \leq t \leq C \log\log n$.
\end{abstract}



\end{frontmatter}

\section{Introduction}
\label{sec:intro}

Cellular automata, introduced by von Neumann \cite{theoryautomata} after a suggestion of Ulam \cite{processestransformations}, are dynamical systems acting on graphs using local and homogeneous update rules. The $H$-bootstrap percolation process is one example of such an automaton and can be described as follows. Given a fixed graph $H$ and a graph $G \subseteq K_n$, set $G_0 = G$ and then, for each $t=0,1,2,\ldots$, let
\begin{equation}
\label{eq:graphBootstrap}
 G_{t+1} = G_t \cup \{e \in E(K_n): \exists H \mbox{ with } e \in H \subseteq G_t \cup {e} \}.
\end{equation} 
Let $\closure{G}{H} = \bigcup_{t=0}^{\infty} G_t$ denote the \textit{closure} of $G$ under $H$-bootstrap percolation. We say that $G$ \textit{percolates}  in the $H$-bootstrap process (or $H$-percolates), if $\closure{G}{H} = K_n$. (See Figure \ref{figure:K_r-BP}).

The notion of $H$-percolation, introduced by Bollob{\'a}s in 1968 \cite{weakSaturation} under the name of \textit{weak saturation}, has been extensively studied in the case where $H$ is a complete graph. Initially, the extremal properties of the $H$-bootstrap process attracted the most attention. Alon \cite{alonSaturation}, Frankl \cite{franklSaturation} and Kalai \cite{kalaiSaturation} independently confirmed a conjecture of Bollob{\'a}s and proved that the smallest $K_r$-percolating graphs on $n$ vertices have size $\binom{n}{2}-\binom{n-r+2}{2}$.

\begin{figure}[htb] \centering
  \begin{tikzpicture}
    \tikzstyle{vertex}=[draw,shape=circle,minimum size=10pt,inner sep=0pt]

    \foreach \name/\x/\y in {1/180, 2/120, 3/60, 4/0, 5/300, 6/240}
      \node[vertex] (P-\name) at (\x:1) {$\name$};
    \foreach \from/\to in {1/2, 5/3,4/2, 1/3, 6/2, 2/3, 1/4, 4/5, 6/5}
      {\draw (P-\from) -- (P-\to);}
    \foreach \from/\to in {4/3}
      {\draw[dashed] (P-\from) -- (P-\to);}
    \draw[black] (0,-1.5) node {$t=0$};
      
    \begin{scope}[xshift=3.5cm]
      \foreach \name/\x/\y in {1/180, 2/120, 3/60, 4/0, 5/300, 6/240}
	\node[vertex] (P-\name) at (\x:1) {$\name$};
      \foreach \from/\to in {1/2, 5/3,4/2, 1/3, 6/2, 2/3, 1/4, 4/5, 6/5}
	{\draw (P-\from) -- (P-\to);}
      \draw[very thick] (P-3) -- (P-4);
      \foreach \from/\to in {1/5, 2/5}
	{\draw[dashed] (P-\from) -- (P-\to);}
    \draw[black] (0,-1.5) node {$t=1$};
    \end{scope}
      
    \begin{scope}[xshift=7cm]
      \foreach \name/\x/\y in {1/180, 2/120, 3/60, 4/0, 5/300, 6/240}
	\node[vertex] (P-\name) at (\x:1) {$\name$};
      \foreach \from/\to in {1/2, 5/3,4/2, 1/3, 6/2, 2/3, 1/4, 4/5, 6/5}
	{\draw (P-\from) -- (P-\to);}
      \foreach \from/\to in {3/4, 1/5, 2/5}
	{\draw[very thick] (P-\from) -- (P-\to);}
      \foreach \from/\to in {1/6, 3/6, 4/6}
	{\draw[dashed] (P-\from) -- (P-\to);}
    \draw[black] (0,-1.5) node {$t=2$};
    \end{scope}
      
    \begin{scope}[xshift=10.5cm]
      \foreach \name/\x/\y in {1/180, 2/120, 3/60, 4/0, 5/300, 6/240}
	\node[vertex] (P-\name) at (\x:1) {$\name$};
      \foreach \from/\to in {1/2, 5/3,4/2, 1/3, 6/2, 2/3, 1/4, 4/5, 6/5}
	{\draw (P-\from) -- (P-\to);}
      \foreach \from/\to in {3/4, 1/5, 2/5, 1/6, 3/6, 4/6}
	{\draw[very thick] (P-\from) -- (P-\to);}
    \draw[black] (0,-1.5) node {$t=3$};
    \end{scope}
    
  \end{tikzpicture}
  \caption{An example of the $K_4$-bootstrap percolation process. Dashed edges are added to the graph on the next time step.}
  \label{figure:K_r-BP}
\end{figure}
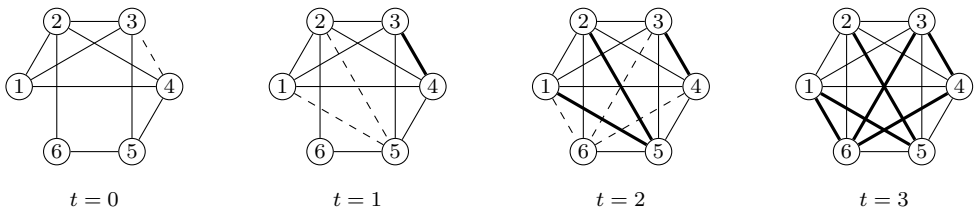

Recently, Bollob{\'a}s \cite{firstGraphBootstrap} observed a strong connection between weak saturation and $r$-neighbour bootstrap percolation, a dynamical process suggested in 1979 by Chalupa, Leath and Reich \cite{bootstrapbethe}. For an integer $r \geq 2 $, the $r$-neighbour bootstrap process on a graph $G = (V,E)$ with an `initial set' of vertices $A \subseteq V$ is defined by setting $A_0 = A$ and for $t = 0,1,2,\ldots$, defining 
\begin{equation}
\label{eq:bootstrapPercolation}
 A_{t+1} = A_t \cup \{v \in V: |N(v) \cap A_t| \geq r \},
\end{equation}
where $N(v)$ is the set of neighbours of $v$ in $G$. The set $\closure{A}{} = \bigcup_{t=0}^{\infty} A_t$ is the closure of $A$ and we say that $A$ percolates if $\closure{A}{} = V$.  Often, the vertices in the set $A_t$ are called `infected' and the remaining vertices are `healthy'. The usual question asked in the context of $r$-neighbour bootstrap percolation is the following: if the vertices of $G$ are initially infected independently at random with probability $p$, for what values of $p$ is percolation likely to occur? The probability of percolation is clearly non-decreasing in $p$ hence it is natural to define the \textit{critical probability} $p_c(G,r)$ as
\begin{equation}
\label{eq:pcrNbr}
 p_c(G,r) = \inf \{ p : \PP_p (\closure{A}{} = V(G)) \geq 1/2 \}.
\end{equation}
The study of critical probabilities has brought numerous and often very sharp results for various graphs $G$ and the values of the infection threshold. For example, van Enter \cite{strayleysargument} and Schonmann \cite{cellularbehaviour} studied $r$-neighbour bootstrap percolation on $\ZZ^d$, Holroyd \cite{sharpmetastability}, Balogh, Bollob{\'a}s and Morris \cite{bootstrapthree}, Balogh, Bollob{\'a}s, Duminil-Copin and Morris \cite{sharpbootstrapall} analysed finite grids, while Balogh and Pittel \cite{randomregular}, Janson, {\L}uczak, Turova and Vallier \cite{bootstrapgnp} and Bollob{\'a}s, Gunderson, Holmgren, Janson and Przykucki \cite{galtonWatson} worked with random graphs.

Motivated by this approach, Balogh, Bollob{\'a}s and Morris \cite{graphBootstrap} defined the critical probability for $H$-bootstrap percolation on $K_n$ to be
\begin{equation}
\label{eq:pcHBtrp}
 p_c(n,H) = \inf \{ p : \PP_p (\closure{G_{n,p}}{H} = K_n) \geq 1/2 \},
\end{equation}
where $G_{n,p}$ is the Erd{\H o}s-R{\'e}nyi random graph, obtained by choosing every edge of $K_n$ independently at random with probability $p$. In \cite{graphBootstrap}, they showed that for all $r \geq 4$, taking $\lambda(r) = \left ( \binom{r}{2}-2 \right ) / (r-2)$ and $n \in \NN$ sufficiently large,
\begin{equation}
\label{eq:pcHBtrpBounds}
  \frac{n^{-1/\lambda(r)}}{2e \log n} \leq p_c(n,K_r) \leq n^{-1/\lambda(r)} \log n.
\end{equation}

In this paper we focus on a different question related to $K_r$-bootstrap percolation. Namely, for what values of $p$ is percolation likely to occur \textit{by time $t$}? Defining $K_r$-bootstrap percolation as in \eqref{eq:graphBootstrap}, let $T = T(n, r, G_0) = \min \{ t: G_t = K_n \text{ in the $K_r$-bootstrap process} \}$. Define the \textit{critical probability for percolation by time $t$} to be 
\begin{equation}
\label{eq:pcHBtrpTime}
 p_c(n,r,t) = \inf \{p : \PP_{p} (T \leq t) \geq 1/2 \}.
\end{equation}

For notational convenience, for any $r \geq 4$ and any $t \geq 1$, set 
\begin{equation}
\label{eq:initialNotation}
\tau = \tau(r) = \binom{r}{2}-1,\ \quad e_t = \tau^t, \text{ and } \quad v_t = (r-2) \frac{\tau^t-1}{\tau-1} + 2.
\end{equation}
 The following theorem is the main result of this paper.
\begin{theorem}
\label{thm:criticalProb}
Let $r \geq 4$ and $t = t(n) \leq \frac{\log \log n}{3\log \tau}$. Let $(p_n)_{n=1}^{\infty}$ be a sequence of probabilities, let $\omega(n) \to \infty$ and let $T = T(n,r,G_{n,p_n})$. Under the $K_r$-bootstrap process,
\begin{enumerate}[(i)]
 \item if, for all $n$, $p_n \geq n^{-(v_t-2)/e_t} \log n$, then $\PP_{p_n} (T \leq t) \to 1$ as $n \to \infty$ and \label{stat:1}
 \item if, for all $n$, $p_n \leq n^{-(v_t-2)/e_t} / \omega (n)$, then $\PP_{p_n} (T \leq t) \to 0$ as $n \to \infty$. \label{stat:2}
\end{enumerate}
\end{theorem}

Thus, Theorem \ref{thm:criticalProb} shows that for all $r \geq 4$ and $1 \leq t \leq \frac{\log\log n}{3 \log \tau}$, and $\omega(n) \to \infty$, for $n$ sufficiently large we have
\begin{equation}\label{eq:crit-prob-bds}
\frac{n^{-(v_t-2)/e_t}}{\omega (n)} \leq p_c(n, r, t) \leq n^{-(v_t-2)/e_t} \log n.
\end{equation}

Similar questions related to the time of $r$-neighbour bootstrap percolation on grids have recently been studied by Bollob{\'a}s, Holmgren, Smith and Uzzell \cite{densebootstraptime}, Bollob{\'a}s, Smith and Uzzell \cite{densebootstraptimeall} and by Balister, Bollob{\'a}s and Smith \cite{timeOfBootstrap}. The time of the $r$-neighbour bootstrap process on the random graph $G(n,p)$ was analysed in \cite{bootstrapgnp}.

Before we continue, let us briefly discuss how the bounds on the critical probability for percolation by time $t$ in Theorem \ref{thm:criticalProb} relate to the bounds on $p_c(n,K_r)$ in \eqref{eq:pcHBtrpBounds}. By the definitions of $p_c(n,K_r)$ in \eqref{eq:pcHBtrp} and of $p_c(n,r,t)$ in \eqref{eq:pcHBtrpTime}, we clearly have $p_c(n,K_r) \leq p_c(n,r,t)$. Ignoring the polylogarithmic factors in the critical probability, we see that the exponent of $n$ in \eqref{eq:crit-prob-bds} can be re-written using the identities in \eqref{eq:initialNotation} as follows, writing $\lambda = \lambda(r)$,
\begin{equation} \label{eq:criticalExponent}
  -\frac{v_t-2}{e_t} = -\frac{(r-2) \frac{\tau^t-1}{\tau-1}}{\tau^t} = -\frac{r-2}{\binom{r}{2}-2} \frac{\tau^t-1}{\tau^t} = - \frac{1}{\lambda} + \frac{1}{\lambda \tau^t}.
\end{equation}
For $t = \frac{\log \log n}{3\log \tau}$, which is the maximum value of $t$ covered by Theorem \ref{thm:criticalProb}, in \eqref{eq:criticalExponent} we obtain
\[
 - \frac{1}{\lambda} + \frac{1}{\lambda \tau^t} = - \frac{1}{\lambda} + \frac{1}{\lambda \tau^{(\log_{\tau} \log n)/3}} = - \frac{1}{\lambda} + \frac{1}{\lambda (\log n)^{1/3}}.
\]
Hence the exponent of $n$ in \eqref{eq:crit-prob-bds} tends to the exponent in \eqref{eq:pcHBtrpBounds} for $t = \frac{\log \log n}{3\log \tau}$. In fact, if one managed to show that \eqref{eq:crit-prob-bds} holds all the way up to $t = \frac{\log \log n}{\log \tau}$ then the bounds in \eqref{eq:pcHBtrpBounds} and \eqref{eq:crit-prob-bds} would match up to polylogarithmic factors.

The proofs of both statements of Theorem \ref{thm:criticalProb} rely on the properties of a family of graphs, denoted $\{F_t : t \geq 1\}$, that are described in detail in Section \ref{sec:FtProperties}.  For each $t$, there is a pair of vertices in $V(F_t)$ so that if $F_t$ occurs as a subgraph of $G_0$, then that pair is guaranteed to be added to the graph by time $t$.  The graph $F_t$ is thought of as `anchored' on that special pair of vertices.  

To prove Statement \eqref{stat:1} of Theorem \ref{thm:criticalProb}, Janson's inequality is used to bound from below the probability that a particular pair $\{x,y\}$ is contained as the anchor vertices in some copy of $F_t$.  To establish a bound in this way, estimates are needed on the probability that two overlapping copies of $F_t$ occur in $G_0$.  This amounts to determining the minimum possible ratio of edges to vertices in some non-trivially overlapping pair.  It turns out that the minimum ratio is not obtained for one of the extreme cases, i.e., neither for two copies of $F_t$ that share only one vertex, nor for two copies that share all but one vertex. Even though we do not prove it directly in this paper, the minimal density of two overlapping copies of $F_t$ is not monotone in the size of their common part and it can be shown that, as $t \to \infty$, the two overlapping copies of $F_t$ that minimise the edge-to-vertex ratio share an approximately $4/((r+1)(r-2))$ proportion of the vertex set. Bounding this ratio from below for all possible configurations of two such copies is the main challenge in the proof of the upper bound on $p_c(n,r,t)$, and is dealt with in detail in Section \ref{sec:overlappingCopies}.

To prove Statement \eqref{stat:2} of Theorem \ref{thm:criticalProb} we employ two extremal results about graphs that add an edge $e$ to the graph in at most $t$ time steps: one of them to bound the number of their vertices from above, and one (a corollary of a highly nontrivial result in \cite{graphBootstrap}) to bound their edge density from below. Then, for $p$ as in Statement \eqref{stat:2} of Theorem \ref{thm:criticalProb}, we show that with high probability no such graph can be found in $G_{n,p}$. This completes the proof of our main result.

The remaining sections of the paper are organised as follows. In Section \ref{sec:r=3} we briefly discuss the $K_3$-bootstrap percolation process which behaves differently than $K_r$-bootstrap processes when $r \geq 4$.  In Section \ref{sec:FtProperties}, we introduce the graphs $F_t$ that are the main focus of the proofs to come and prove some key properties.  In Section \ref{sec:overlappingCopies}, which is the crucial part of our argument, we prove some properties of graphs consisting of two overlapping copies of $F_t$. In Sections \ref{sec:upperBound} and \ref{sec:lowerBound} we prove Statements \eqref{stat:1} and \eqref{stat:2} of Theorem \ref{thm:criticalProb} respectively. Finally, in Section \ref{sec:questions}, some open problems are stated.

\section{$K_3$-bootstrap percolation}
\label{sec:r=3}

In this section we discuss the special case of the $K_3$-bootstrap process. Observe that a graph $G$ percolates in $K_3$-bootstrap percolation if and only if $G$ is connected. Also, at every time step each non-edge between vertices at distance $2$ is added to the graph. Therefore, if $G$ is a connected graph with diameter $d$, then the diameter of the graph obtained from $G$ after one step of the $K_3$-bootstrap process is $\lceil d/2 \rceil$. Hence, $G$ percolates in $\lceil \log_2 d \rceil$ time steps.

The diameter of random graphs was investigated by Bollob{\'a}s \cite{randomDiameter} who proved the following theorem.
\begin{theorem}
\label{thm:diameterOfGnp} Let $G_{n,p}$ be the Erd{\H o}s-R{\'e}nyi random graph.
 \begin{enumerate}
  \item \label{item:diameterOfGnp1} Suppose $p^2n - 2 \log n \to \infty$ and $n^2 (1 - p) \to \infty$. Then $G_{n,p}$ has diameter $2$ whp.
  \item \label{item:diameterOfGnp2} Suppose the functions $d = d(n) \geq 3$ and $0 < p = p(n) < 1$ satisfy $(\log n)/d-3 \log \log n \to \infty$, $p^d n^{d-1} - 2 \log n \to \infty$ and $p^{d-1} n^{d-2} - 2 \log n \to -\infty$. Then $G_{n,p}$ has diameter $d$ whp.
 \end{enumerate}
\end{theorem}

In order to re-phrase this result in the form of intervals for $p$ in which the diameter is constant with high probability, let $\omega(n) = o (\log n)$ tend to infinity arbitrarily slowly. Clearly, if $p \geq 1-1/(n^2 \omega(n))$ then whp. $G_{n,p} = K_n$ which has diameter $1$. Simplifying a bit, Theorem \ref{thm:diameterOfGnp} implies that if
\[
 \sqrt{\frac{2 \log n + \omega(n)}{n}} \leq p \leq 1-\frac{\omega(n)}{n^2}
\]
then $G_{n,p}$ has diameter $2$, and that for $3 \leq d \leq \log n / 4 \log \log n$, if
\[
p(n) \in \left ( (2\log n + \omega(n))^{\frac{1}{d}} n^{-\frac{d-1}{d}}, (2\log n - \omega(n))^{\frac{1}{d-1}} n^{-\frac{d-2}{d-1}} \right )
\]
then the random graph $G_{n,p(n)}$ has diameter $d$ with high probability. This answers the question about the time of $K_3$-bootstrap percolation.

\section{Adding an edge to the graph using sparse subgraphs}
\label{sec:FtProperties}

Throughout the following sections, fix $r \geq 4$.  As $r$ is fixed, for simplicity, it is often omitted from the notation.  We define a family $\{F_t: t \geq 1\}$ of graphs that add a given pair as an edge to the graph in the $K_r$-bootstrap process exactly at time $t$. We prove that these are the ``sparsest'' minimal such graphs (i.e., they minimise the ratio of the number of edges to the number of vertices). The main result in this section, presented in Section \ref{sec:overlappingCopies}, is a lower bound on the edge-density of two non-disjoint copies of the graph $F_t$. This bound is the key element of arguments for the proof of Statement (i) of Theorem \ref{thm:criticalProb}.

The graph $F_t$ is defined recursively and the fixed edge that is added to the graph at time $t$ using $F_t$ will always be denoted by $e_0 = \{1,2\}$.

For $t = 1$, set $F_1 = K_r - e_0$, an $r$-clique missing one edge. 

For each $t \geq 1$, given $F_t$, for each $e \in E(F_t)$, let $V(e)$ be a set of $r-2$ new vertices and let $K(e)$ be a copy of $K_r-e$, an $r$-clique missing one edge, on vertex set $V(e) \cup e$.  Then, $F_{t+1}$ is defined to be the graph with vertex set 
\[
V(F_{t+1}) = V(F_t) \cup \left( \bigcup_{e \in E(F_t)} V(e) \right)
\]
and edge set 
\[
E(F_{t+1}) = \bigcup_{e \in E(F_t)} E(K(e))
\]
(see Figure \ref{figure:constructionOfF_t}). Note that any edge $e \in E(F_{t+1})$ is incident to at least one vertex in $V(f)$ for some $f \in E(F_t)$, i.e., to a vertex in $V(F_{t+1}) \setminus V(F_t)$.

Recall that we define $\tau = \binom{r}{2}-1$ and the numbers $e_t$ and $v_t$ in equation \eqref{eq:initialNotation}.  By induction on $t$, it can be shown that for every $t \geq 1$, the number of edges and vertices in the graph $F_t$ are given by
\begin{align}
\label{eq:e_t}
 e_t & = e(F_t) = |E(F_t)| = \tau^t \text{ and}\\
\label{eq:v_t}
 v_t & = v(F_t) = |V(F_t)| = |V(F_{t-1})| + e_{t-1} (r-2) = 2 + (r-2) \frac{\tau^t - 1}{\tau -1}. 
\end{align}

\begin{figure}[htb] \centering
  \begin{tikzpicture}
    \tikzstyle{vertex}=[draw,shape=circle,minimum size=5pt,inner sep=0pt]

    \foreach \name/\x/\y in {1/0/0, 2/0/1, 3/1/1, 4/1/0}
      \node[vertex] (P-\name) at (\x,\y) {~};
    \draw (0,0.5) node [left] {$e_0$};
    \foreach \from/\to in {1/3, 1/4, 2/3, 2/4, 3/4}
      {\draw (P-\from) -- (P-\to);}
    \foreach \from/\to in {1/2}
      {\draw[dashed] (P-\from) -- (P-\to);}
    \draw[black] (0.5,-2) node {$F_1$};
      
    \begin{scope}[xshift=2.25cm]

    \foreach \name/\x/\y in {1/0/0, 2/0/1, 3/1/1, 4/1/0, 231/0/2, 232/0.85/2, 241/1.4/1.9, 242/1.9/1.4, 341/2/0.9, 342/2/0.1, 131/1.4/-0.9, 132/1.9/-0.4, 141/0/-1, 142/0.85/-1}
      \node[vertex] (P-\name) at (\x,\y) {~};
    \draw (0,0.5) node [left] {$e_0$};
    \foreach \from/\to in {1/131, 1/132, 1/141, 1/142, 2/231, 2/232, 2/241, 2/242, 3/131, 3/132, 3/231, 3/232, 3/341, 3/342, 4/141, 4/142, 4/241, 4/242, 4/341, 4/342, 131/132, 141/142, 231/232, 241/242, 341/342}
      {\draw (P-\from) -- (P-\to);}
    \foreach \from/\to in {1/2}
      {\draw[dashed] (P-\from) -- (P-\to);}
    \draw[black] (1,-2) node {$F_2$};
    
    \end{scope}
    
    \draw (5,0.5) node {$\cdots$};
    
    \begin{scope}[xshift=6cm]

    \draw (4.25,0.5) ellipse (4.5cm and 1.75cm);
    \foreach \name/\x/\y in {1/0.25/0, 2/0.25/1, 3/2.5/1.25, 4/3.5/-0.7, 5/5.2/1.5, 6/5.3/0.5, 561/6.55/1.7, 562/6.75/1.1, 461/6.55/-0.7, 462/6.75/-0.1}
      \node[vertex] (P-\name) at (\x,\y) {~};
    \foreach \from/\to in {1/2}
      {\draw[dashed] (P-\from) -- (P-\to);}
    \draw (0.25,0.5) node [left] {$e_0$};
    \draw[black] (4.25,-2) node {$F_{t}$};
    \foreach \x/\y in {1.25/0.35, 1.25/0.1, 1.25/-0.1, 1.25/-0.3}
      \draw (P-1) -- (\x,\y);
    \draw (1.25,0) node [right] {$\Bigg\} (r-2)(r-1)^{t-1}$};
    \draw [dashed] (1.75,-0.955) arc (323.96:345:2.473);
    \draw [dashed] (1.75,1.955) arc (36.04:-5:2.473);
    \foreach \x/\y in {3.5/1.5, 3.5/1.25, 3.5/1}
      \draw (P-3) -- (\x,\y);
    \foreach \x/\y in {4.5/-0.25}
      \draw (P-4) -- (\x,\y);
    \draw [dashed] (3.75,-1.239) arc (-23.5:23.5:4.362);
    \draw (4.5,0.5) node {$\cdots$};
    \draw [dashed] (4.5,-1.239) arc (-23.5:23.5:4.362);
    \draw (5.5,-1.181) arc (-21.94:21.94:4.5);
    \draw[black] (7.25,0.5) node {$V(F_{t}) \setminus V(F_{t-1})$};
    \foreach \x/\y in {5/561, 5/562, 6/561, 6/562, 561/562}
      \draw (P-\x) -- (P-\y);
    \foreach \x/\y in {4/461, 4/462, 6/461, 6/462, 461/462}
      \draw (P-\x) -- (P-\y);
    \end{scope}
    
  \end{tikzpicture}
  \caption{Construction of the graph $F_t$. Note that every edge in $F_{t}$ is incident to at least one vertex in $V(F_{t}) \setminus V(F_{t-1})$.}
  \label{figure:constructionOfF_t}
\end{figure}
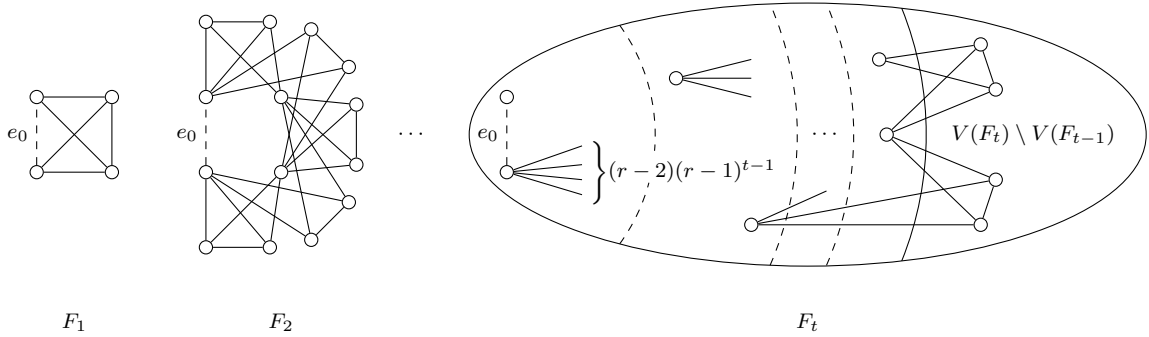

\begin{lemma}
\label{lem:FtInfects}
In the $K_r$-bootstrap process started from $F_t$ the edge $e_0$ is added to the graph in exactly $t$ steps.
\end{lemma}
\begin{proof}
 We prove this lemma by induction on $t$. The statement is trivial for $t=1$ as $F_1 = K_r - e_0$. Assume that the Lemma holds for $t=k \geq 1$. Note that after one step of the process started from $F_{k+1}$ we obtain a copy of $F_{k}$ in our graph since $F_{k+1}$ is obtained from $F_{k}$ by placing a copy of $K_r$ minus an edge on every edge of $F_k$. Thus $e_0$ is added to the graph after at most $k+1$ steps of the process started from $F_{k+1}$.

The construction of $F_{k+1}$ can be also seen as placing a copy of $F_k$ on each of the $\tau$ edges of $F_1 = K_r - e_0$. By induction we know that these copies of $F_k$ on their own add the respective edges of $F_1$, their \textit{anchor} edges, in $k$ time steps. This process could possibly accelerate if some interaction between two different copies of $F_k$ occurred early in the process, say, before time $k$ when the $F_k$'s on their own add their respective anchor edges. Therefore, let $F^1, F^2, \ldots, F^{\tau}$ be these different copies of $F_k$ in $F_{k+1}$. By construction of $F_{k+1}$ we have that for all $i \neq j$:
 \begin{enumerate}
  \item $F^i$ and $F^j$ share at most one vertex,
  \item no vertex of $F^i$ other than the vertices in its anchor edge can have a neighbour outside $F^i$.
 \end{enumerate}
 
 For every $s \in \{0,1,\ldots, k\}$, and $i \in \{1,2, \ldots,\tau\}$, let $E_{i,s}$ be the edges within $V(F^{i})$ in the $s$-th step of the $K_r$-bootstrap process on the subgraph induced by $F^i$ and let $E_s$ be the edges in $F_{k+1}$ in the $s$-th step of the $K_r$ bootstrap process.  The aim is to show that for $s \leq k$, there is no interaction between the processes in each $F^i$ so that $E_s = \bigcup_{i = 1}^{\tau} E_{i, s}$.  This is proved by induction on $s$.  Note that for $s \leq k-1$, if it is true that $E_s = \bigcup_{i = 1}^{\tau} E_{i, s}$, then any triangle in the graph at that time is contained entirely within the vertex set of one of the $F^i$.  Also, recall that $E_{i, s}$ contains only edges within $V(F^i)$.
 
 For $s = 0$, the statement $E_0 = \bigcup_{i = 1}^{\tau} E_{i, 0}$ is true by construction.
 
 Suppose that for some $s \in [0, k-1]$, that $E_s = \bigcup_{i = 1}^{\tau} E_{i, s}$ and that there is an edge $\{w_1, w_2\} \in E_{s+1} \setminus \bigcup_{i = 1}^{\tau} E_{i, s+1}$.  As the edge $\{w_1, w_2\}$ must have been added at time $s+1$, let $w_3, \ldots, w_r$ be the $r-2$ other vertices in the copy of $K_r$ minus an edge that witness $\{w_1, w_2\}$ being added at time $s+1$.  Then, at time $s$, for any $i, j > 2$, the vertices $w_1, w_i, w_j$ form a triangle, as do the vertices $w_2, w_i, w_j$.  Then, by the induction hypothesis, there is an $i_0$ so that $w_1, w_2, w_i, w_j \in V(F^{i_0})$.  Again, by hypothesis, this means that the edges $\{w_1, w_i\}, \{w_2, w_i\},  \{w_1, w_j\}, \{w_2, w_j\}, \{w_i, w_j\}$ are all contained in $E_{i_0, s}$.  This is a contradiction, as then the edge $\{w_1, w_2\}$ would have been added within the graph $F^{i_0}$ at time $s+1$.  This completes the induction on $s$.

Thus the first $k$ steps of the process started from $F_{k+1}$ look like $\tau$ independent processes started from $F_k$'s. Hence $e_0$ is not added to the graph by time $k$. This completes the proof of the lemma.
\end{proof}

Recall that we denote
\begin{equation}\label{eq:lambda}
 \lambda = \frac{\binom{r}{2}-2}{r-2} = \frac{r+1}{2} - \frac{1}{r-2}
\end{equation}
and that $\tau$ can be written as
\begin{equation}\label{eq:tau}
 \tau = \binom{r}{2} - 1 = \frac{(r+1)(r-2)}{2}.
\end{equation}
Note that \eqref{eq:lambda} and \eqref{eq:tau} imply that
\begin{equation}\label{eq:tauAndLambda}
 (r-2) \lambda = \tau -1.
\end{equation}
Define for convenience,
\begin{equation}\label{eq:c_t}
c_t = \frac{1}{\tau^t-1},
\end{equation}
which is used to simplify certain expressions involving $\frac{e_t}{v_t-2}$ as in \eqref{eq:crucialRatio} below.
By \eqref{eq:v_t} and \eqref{eq:tauAndLambda} we have
\begin{equation} \label{eq:v_t-2}
 v_t -2 = (r-2) \frac{\tau^t - 1}{\tau -1} = \frac{\tau^t - 1}{\lambda}.
\end{equation}
Hence, using \eqref{eq:e_t} and the above relations,
\begin{equation}
\label{eq:crucialRatio}
  \frac{e_t}{v_t-2} = \frac{\tau^t}{\frac{\tau^t-1}{\lambda}} = \lambda \left ( 1 + c_t \right ) = \lambda + \frac{\tau-1}{r-2} \frac{1}{\tau^t-1} = \lambda + \frac{1}{v_t-2}.
\end{equation}
Equation \eqref{eq:crucialRatio} is used throughout this section to show that $F_t$ is the sparsest minimal graph that adds $e_0$ to the graph in $t$ time steps of the $K_r$-bootstrap process.

Let us recall the following Witness-Set Algorithm introduced in \cite{graphBootstrap}. Given a graph $G$, we assign a subgraph $F = F(e) \subseteq G$ to each edge $e \in \closure{G}{K_r}$ as follows:
\begin{enumerate}
 \item If $e \in G$ then set $F(e) = \{e\}$.
 \item Choose an order in which to add the edges of $\closure{G}{K_r}$, and at each step identify which $r$-clique was completed (if more than one is completed then choose one).
 \item Add the edges one by one. If $e$ is added by the $r$-clique $K$, then set
\[
 F(e) := \bigcup_{e \neq e' \in K} F(e').
\]
\end{enumerate}
A graph $F$ is an \textit{$r$-witness set} if there exists a graph $G$, an edge $e$, and a realization of the Witness-Set Algorithm (i.e., a sequence of choices as in Step 2) such that $F = F(e)$. The following highly nontrivial extremal result occurs as Lemma 9 in \cite{graphBootstrap}, which is stated here without repeating the proof.
\begin{lemma}
 \label{lem:witnessManyEdges}
Let $F$ be a graph and $r \geq 4$, and suppose that $F$ is an $r$-witness set. Then,
\[
 |E(F)| \geq \lambda (|V(F)|-2) + 1.
\]
\end{lemma}

We say that a graph $G$ is a \textit{minimal graph adding $e$} if $e \in \closure{G}{K_r}$ but for all proper subgraphs $G' \subsetneq G$ of $G$ we have $e \notin \closure{G'}{K_r}$. It's an immediate observation that every minimal graph adding $e$ to $G$ is an $r$-witness set. Hence we have the following corollary.
\begin{corollary}
 \label{cor:manyEdges}
Let $r \geq 4$ and let $F$ be a minimal graph adding $e$ to the graph for some $e \in \closure{F}{K_r}$. Then
\begin{equation}
\label{eq:manyEdges}
 |E(F)| \geq \lambda (|V(F)|-2) + 1.
\end{equation}
\end{corollary}

We now show that $F_t$ maximises the number of vertices among all minimal graphs that add $e_0$ to the graph in exactly $t$ time steps of the $K_r$-bootstrap process.
\begin{lemma}
 \label{lem:minimalsSmall}
 Let $r \geq 4$, $t \geq 1$ and let $F$ be a minimal graph adding $e_0$ to the graph at time $t$ in the $K_r$-bootstrap process. Then $|V(F)| \leq v_t = (r-2) \frac{\tau^t-1}{\tau-1}+2$ and $|E(F)| \leq e_t = \tau^t$.
\end{lemma}

\begin{proof}
 We prove the lemma by induction on $t$. For $t=1$ the lemma is trivial as $K_r - e_0$ is the only minimal graph adding $e_0$ at the first time step. Hence assume that the lemma holds for some $t \geq 1$ and consider a minimal graph $F$ such that $e_{0}$ is added at time $t+1$ in the $K_r$-bootstrap process started from $F$.
 
 After one step of the process we obtain a graph $F'$ containing some minimal subgraph $F''$ that adds $e_{0}$ in $t$ additional time steps. By induction we have $|V(F'')| \leq (r-2) \frac{\tau^t-1}{\tau-1}+2$ and $|E(F'')| \leq \tau^t$. Now, since $F$ was a minimal graph adding $e_{0}$ in time $t+1$, to maximise the number of vertices and edges in $F$ we should in the first step of the process add every edge $e$ of $F''$ using a copy of $K_r - e$ disjoint from the copies adding other edges in $F''$. This shows that $|E(F)| \leq \tau |E(F'')|$ and $|V(F)| \leq |V(F'')| + (r-2) |E(F'')|$. This completes the induction and the lemma follows.
 \end{proof}
 
 A proof closely following that of Lemma \ref{lem:minimalsSmall} immediately shows a further extremal result.

\begin{corollary}
 \label{cor:largestIsomorphic}
For any $t \geq 1$, up to isomorphism, $F_t$ is the only minimal graph on $v_t$ vertices adding $e_{0}$ to the graph in exactly $t$ time steps.
\end{corollary}

As usual, for any graph $G$ and $A,B \subseteq V (G)$, let $E(A,B) =  \left \{ \{a,b\} \in E(G) : a \in A, b \in B \right \} $ and $e(A, B) = |E(A, B)|$. We shall often take $B = V(G)$ and, to simplify the notation, we shall write $e(A, G)$ to denote $e(A, V(G))$. Note that $e(A, G)$ is the number of edges of $G$ with at least one endpoint in $A$.

In the proofs to come, results on edge-densities of subsets of the graphs $\{F_t : t \geq 1\}$ are proved by induction on $t$.  To make the notation clearer, let us use $E_t(A, B)$ to denote the edges between $A$ and $B$ in the graph $F_t$ and $e_t(A, B) = |E_t(A, B)|$.  As usual, $\delta(G) = \min \{\deg_G(v) : v \in V(G) \}$ is used for the minimum degree of $G$. We shall find the following simple estimate useful in our examination of $F_t$.
\begin{lemma}
 \label{lem:degree}
 For any $t \geq 2$ and any set $L \subseteq V(F_t)$
 \[
  e_t(L, F_t) \geq \frac{r-1}{2} |L|.
 \]
\end{lemma}

\begin{proof}
 Note that for $t \geq 2$, $\delta(F_t) = r-1$.  Thus
 \[
  (r-1) |L| \leq \sum_{v \in L} \deg_{F_t}(v) = 2e_t(L, L) + e_t(L, L^c) \leq 2e_t(L, F_t).
 \]
\end{proof}

\subsection{Overlapping copies of $F_t$}
\label{sec:overlappingCopies}

To prove Statement \eqref{stat:1} of Theorem \ref{thm:criticalProb} we shall show that if $p$ is large enough then with high probability there is a copy of $F_t$ anchored on every pair of vertices in $G_{n,p}$. Towards this aim, we shall show that a measure of the variance of the number of such copies of $F_t$ anchored on a fixed edge $e_0$ is not too large compared to their expected number. In Section \ref{sec:upperBound}, this fact together with Janson's inequality (see Theorem \ref{thm:JansonIneq}) is used to deduce the desired result. Hence, we need to prove that it is significantly ``harder'' (in terms of the ratio of the number of edges to the number of vertices) to find two different such copies of $F_t$ that overlap in at least one vertex (other than $1,2 \in e_0$) than it is to find two disjoint such copies.

In particular, as the main result in this subsection, it is shown that for any $L \subseteq V(F_t) \setminus \{1,2\}$,
\begin{equation}\label{eq:hoped-for-lb}
\frac{e_t(L, F_t)}{|L|} \geq \frac{e_t}{v_t-2}
\end{equation}
with equality only when $L = V(F_t) \setminus \{1,2\}$. With this in mind, define $\varepsilon_t$ to be such that
\begin{equation}\label{eq:eps}
1 + \varepsilon_t =  \left ( \frac{v_t-2}{e_t} \right ) \min \left \{ \frac{e_t(L, F_t)}{|L|} : L \subsetneq V(F_t) \setminus \{1,2\} \right \}.
\end{equation}

From the definition above, there is no guarantee that $\varepsilon_t$ is non-negative. Using induction on $t$, we shall prove that this is the case by first giving a weak upper bound on $\varepsilon_t$ in Lemma \ref{lem:epsilon_tUpper} and then using it to prove a relatively sharp lower bound on $\varepsilon_t$ for all $t \geq 1$.

\begin{lemma}\label{lem:epsilon_tUpper}
	For all $r \geq 4$ and $t \geq 1$ we have $\varepsilon_{t} \leq \frac{1}{r+1}$.
\end{lemma}

\begin{proof}
First consider the case $t = 1$. For all $L \subsetneq V(F_1) \setminus \{1,2\} $ the vertices in $F_1 \setminus L$ induce a clique minus the edge joining $1$ and $2$. Hence, for $1 \leq \ell \leq r-3$ and $|L| = \ell$, the vertices in $F_1 \setminus L$ induce $\binom{r-\ell}{2}-1$ edges, which gives
	\[
	\begin{split}
		\frac{e_1(L, F_1)}{|L|} & = \frac{1}{\ell} \left ( \binom{r}{2} - 1 - \binom{r-\ell}{2}+1 \right ) \\
		& = \frac{1}{2\ell} \left ( r^2 - r - r^2 + 2r\ell - \ell^2 + r - \ell \right ) \\
		& = \frac{2r - \ell - 1}{2} \\
		& \geq \frac{r+2}{2} \text{,}
	\end{split}
	\]
	with equality for $\ell=r-3$. Hence, by \eqref{eq:eps}
	\begin{align*}
		\varepsilon_1 &= \left(\frac{r+2}{2}\right) \left(\frac{v_1-2}{e_1}\right) - 1 \\
					&=\left(\frac{r+2}{2} \right)\left(\frac{r-2}{\binom{r}{2}-1} \right) - 1 \\
					&=(r+2)\left(\frac{r-2}{r^2-r-2}\right) - 1 \\
					&= \frac{r+2}{r+1} - 1 = \frac{1}{r+1},
	\end{align*}
which proves the lemma for $t=1$.
	
Consider now the case when $t \geq 2$. By the construction of the graph $F_t$, there are vertices in $V(F_t) \setminus \{1, 2\}$ connected to $1$ and not to $2$. Let $v$ be such a vertex and choose $L = V(F_t) \setminus \{1,2,v\}$ so that $|L| = v_t - 3$ and $e_t(L,F_t) = e_t - 1$ (recall that $e_t(L,F_t)$ counts all edges in $F_t$ with at least one end in $L$, thus here it only misses $\{1,v\}$). Hence, for $t \geq 2$ and $r \geq 4$, from the definition in \eqref{eq:eps} we have
\begin{align*}
		\varepsilon_{t} & \leq \left(\frac{v_t-2}{e_t}\right) \left(\frac{e_t-1}{v_t-3}\right) - 1 \\
		& = \frac{e_t v_t -2 e_t - v_t + 2 - e_t v_t + 3 e_t}{e_t(v_t-3)}  \\
		& = \frac{e_t - v_t + 2}{e_t(v_t-3)}  \\
		& < \frac{1}{v_t-3}  \leq \frac{1}{v_2 - 3}\\
		& = \frac{1}{(r-2)\binom{r}{2} - 1} && \text{(by \eqref{eq:v_t} and \eqref{eq:tau})}\\
		&\leq \frac{1}{(r-2)6 -1} \leq \frac{1}{r+1}.
\end{align*}
This completes the proof of Lemma \ref{lem:epsilon_tUpper}.
\end{proof}

The following lemma gives us another result in a similar direction and is used in this section to show that one need only consider certain choices for $L \subseteq V(F_t)$ in order to determine $\varepsilon_t$.

\begin{lemma}\label{lem:minRatioUpper}
	For all $r \geq 4$ and $t \geq 2$ we have
	\[
	 \min \left \{ \frac{e_t(L, F_t)}{|L|} : L \subsetneq V(F_t) \setminus \{1,2\} \right \} < \frac{r+1}{2}.
	\]
\end{lemma}
\begin{proof}
 We prove the lemma by giving an example of a simple set $L \subsetneq V(F_t) \setminus \{1,2\}$ that satisfies the inequality. Let $v \in V(F_{t-1}) \setminus V(F_{t-2})$. Then, let
 \[
  L = \{v\} \cup \left \{ u \in V(F_{t}) \setminus V(F_{t-1}) : \{v,u\} \in E(F_{t}) \right \}.
 \]
 Since $v \in V(F_{t-1}) \setminus V(F_{t-2})$, we have $\deg_{F_{t-1}}(v) = r-1$. In $F_t$, every edge of $F_{t-1}$ is replaced with a copy of $K_r$ minus that edge. Hence for every edge adjacent to $v$ in $F_{t-1}$ there are $r-2$ neighbours of $v$ in $F_t$. This implies
 \[
 |L| = 1+\deg_{F_t}(v) = 1+(r-2) \deg_{F_{t-1}}(v)  = 1+(r-1)(r-2).
 \]
To find $e_{t}(L,F_t)$, i.e., the number of edges adjacent to at least one vertex in $L$, we notice that all these edges are contained in the $\deg_{F_{t-1}}(v) = r-1$ copies of $K_r$ minus an edge, that are placed on the edges adjacent to $v$ in $F_{t-1}$ to construct the graph $F_t$. Therefore we have
 \[
  e_{t}(L,F_t) = (r-1) \left ( \binom{r}{2}-1 \right ) = \frac{(r-1)(r+1)(r-2)}{2}.
 \]
This implies that
 \[
  \frac{e_{t}(L,F_t)}{|L|} = \frac{(r+1)}{2} \frac{(r-1)(r-2)}{1+(r-1)(r-2)} < \frac{(r+1)}{2}.
 \]
\end{proof}

We observe that, in general, Lemma \ref{lem:minRatioUpper} yields a worse upper bound on $\varepsilon_t$ than that given by Lemma \ref{lem:epsilon_tUpper}, but its form is useful in the proof of Lemma \ref{claim:full-blobs} to come.

The next theorem is the main tool in the proof of Statement \eqref{stat:1} of Theorem \ref{thm:criticalProb} to come. Here we give a lower bound on $\varepsilon_t$ that holds for all $t \geq 1$.

\begin{theorem}\label{thm:epsilon_tLower}
	For all $r \geq 4$ and $t \geq 1$,
	\begin{equation}\label{eq:epsilon_t}
		\varepsilon_{t} \geq \frac{1}{r+1} \left ( \frac{2}{r^2-2} \right )^{t-1}.
	\end{equation}
\end{theorem}

Before proving Theorem \ref{thm:epsilon_tLower}, a few auxiliary lemmas are stated and proved below, along with an outline of the proof.  These are then used to establish Theorem \ref{thm:epsilon_tLower} by induction on $ t $.

Recall that the graph $F_{t+1}$ is constructed by placing an independent copy of $K_r-e$ on every edge $e$ of $F_t$. Further recall that for each $e \in E(F_t)$ we write $V(e)$ to denote this new set of $r-2$ vertices.

For the induction step from $t$ to $(t+1)$ we will fix a set $L_t \subseteq V(F_t) \setminus \{1,2\}$ and look for the smallest possible edge densities $\frac{e_{t+1}(L_t \cup M, F_{t+1})}{|L_t \cup M|}$ among sets of the form $L_t \cup M$ where $M \subseteq V(F_{t+1}) \setminus V(F_t)$ (see Figure \ref{figure:subsetsOfF_t}).

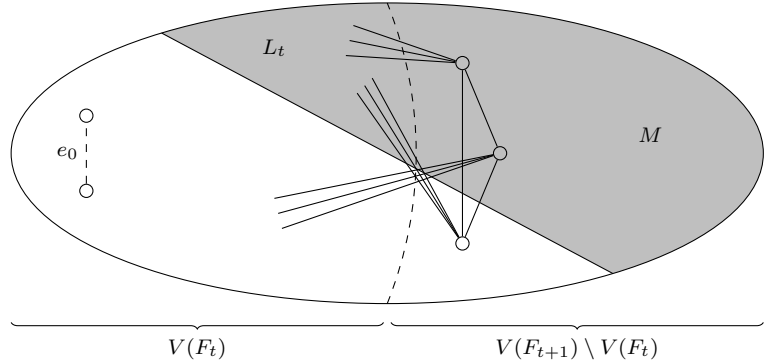
\begin{figure}[htb] \centering
  \begin{tikzpicture}
    \tikzstyle{vertex}=[draw,shape=circle,minimum size=5pt,inner sep=0pt]
    
    \begin{scope}
      \clip (0,0) ellipse (5cm and 2cm);
      \fill[color = lightgray] (-6,3.2) -- (6,3.2) -- (6,-3.2) -- cycle;
    \end{scope}
    \draw (0,0) ellipse (5cm and 2cm);
    \foreach \name/\x/\y in {1/-4/-0.5, 2/-4/0.5, 3/1/1.2, 4/1/-1.2, 5/1.5/0}
      \node[vertex] (P-\name) at (\x,\y) {~};
    \draw (-4,0) node [left] {$e_0$};
    \foreach \from/\to in {1/2}
      {\draw[dashed] (P-\from) -- (P-\to);}
    \draw [dashed] (0,-2) arc (-21.8:21.8:5.385);
    \draw (-3, 1.6) -- (3,-1.6);
    \draw[black] (-1.5,1.4) node {$L_{t}$};
    \draw[black] (3.5,0.25) node {$M$};
    
    \foreach \x/\y in {-0.45/1.7, -0.5/1.5, -0.55/1.3}
      \draw (P-3) -- (\x,\y);
    \foreach \x/\y in {-0.2/1, -0.3/0.9, -0.4/0.8}
      \draw (P-4) -- (\x,\y);
    \foreach \x/\y in {-1.4/-1, -1.45/-0.8, -1.5/-0.6}
      \draw (P-5) -- (\x,\y);
    \foreach \x/\y in {3/4, 3/5, 4/5}
      \draw (P-\x) -- (P-\y);
      
    \draw [decoration={brace, mirror}, decorate] (-5,-2.25) -- (-0.05,-2.25) node [pos=0.5,anchor=north,yshift=-0.1cm] {$V(F_t)$}; 
    \draw [decoration={brace, mirror}, decorate] (0.05,-2.25) -- (5,-2.25) node [pos=0.5,anchor=north,yshift=-0.1cm] {$V(F_{t+1}) \setminus V(F_t)$};
    
  \end{tikzpicture}
  \caption{Sets $L_t$ and $M$ in $F_{t+1}$ together with the edges counted in $e_{t+1}(L_t \cup M, F_{t+1})$.}
  \label{figure:subsetsOfF_t}
\end{figure}

In the following lemma we first deal with the case $L_t = \emptyset$, showing that no set contained entirely in $V(F_{t+1}) \setminus V(F_t)$ can minimise the edge density.

\begin{lemma}\label{lem:L_tEmpty}
	For all $r \geq 4$ and $t \geq 2$ we have
	\[
	 \min \left \{ \frac{e_t(L, F_t)}{|L|} : L \subseteq V(F_t) \setminus V(F_{t-1}) \right \} \geq \frac{r+1}{2}.
	\]
\end{lemma}
\begin{proof}
 Let $L \subseteq V(F_t) \setminus V(F_{t-1})$. Hence for every $v \in L$ we have $\deg_{F_t}(v) = r-1$ and at most $r-3$ neighbours of $v$ are also in $L$. Thus $e_t(L, F_t)$, the number of edges incident to at least one vertex in $L$, satisfies
 \[
e_t(L, F_t) \geq  |L| \left ( \frac{r-3}{2} + 2 \right ) =  |L| \frac{r+1}{2}
 \]
 and the lemma follows.
\end{proof}

Thus, by Lemma \ref{lem:minRatioUpper}, the minimum in equation \eqref{eq:eps} is not attained for any $L \subseteq V(F_t) \setminus V(F_{t-1})$. We now show that if for some $e \in E(F_t)$ certain conditions are fulfilled, then moving all vertices from $V(e)$ into $M$ does not increase the edge-density minimised in equation \eqref{eq:eps}. The details of this are given in the following lemma.

\begin{lemma}\label{claim:full-blobs}
	Let $t \geq 1$, $L_{t} \subseteq V(F_{t}) \setminus \{1,2\}$ with $L_t \neq \emptyset$, and $ M \subseteq V(F_{t+1}) \setminus V(F_t) $. For every edge $e = \{x,y\} \in E(F_{t})$ such that at least one of $x,y$ is in $L_{t}$ we have 
	\[
	\frac{e_{t+1}(L_{t} \cup M, F_{t+1})}{|L_{t} \cup M|} \geq \frac{e_{t+1}(L_{t} \cup M \cup V(e), F_{t+1})}{|L_{t} \cup M \cup V(e)|}.
	\]
\end{lemma}

\begin{proof}
	Let us recall that, by the construction of $F_{t+1}$, all edges incident to vertices in $V(e)$ are either induced by $V(e)$ or they connect the vertices of $V(e)$ to either $x$ or $y$.
	
	Set $m = |V(e) \setminus M|$ and note that the conclusion is trivially true for $ m = 0 $. Thus, assume that $1 \leq m \leq r-2$. 
	
	We consider two different cases. Suppose first that both $x, y \in L_{t}$. Then $e_{t+1}(L_{t} \cup M, F_{t+1})$ already counts all edges incident to a vertex in $V(e) \setminus M$, except those that are induced by this set (because the remaining edges are either incident to $x,y \in L_t$ or to some vertex in $V(e) \cap M$). Since $V(e)$ induces a clique in $F_{t+1}$, we have $e_{t+1}(L_{t} \cup M \cup V(e), F_{t+1}) - e_{t+1}(L_{t} \cup M, F_{t+1}) = \binom{m}{2}$. Therefore
	\begin{align*}
		\frac{e_{t+1}(L_{t} \cup M \cup V(e), F_{t+1})}{|L_{t} \cup M \cup V(e)|}
		& = \frac{e_{t+1}(L_{t} \cup M, F_{t+1}) + \binom{m}{2}}{|L_{t} \cup M| + m}\\
		&=\frac{e_{t+1}(L_{t} \cup M, F_{t+1})}{|L_{t} \cup M|} \cdot \frac{|L_{t} \cup M|}{|L_{t} \cup M|+m} + \frac{\binom{m}{2}}{m} \frac{m}{|L_{t} \cup M| + m}\\
		& = \frac{e_{t+1}(L_{t} \cup M, F_{t+1})}{|L_{t} \cup M|} \cdot \frac{|L_{t} \cup M|}{|L_{t} \cup M|+m} + \frac{m-1}{2} \frac{m}{|L_{t} \cup M| + m}\\
		& \leq \frac{e_{t+1}(L_{t} \cup M, F_{t+1})}{|L_{t} \cup M|} \cdot \frac{|L_{t} \cup M|}{|L_{t} \cup M|+m} + \frac{r-3}{2} \frac{m}{|L_{t} \cup M| + m}.
	\end{align*}
	
	The above expression is a convex combination of $\frac{r-3}{2}$ and $\frac{e_{t+1}(L_{t} \cup M, F_{t+1})}{|L_{t} \cup M|}$.  By Lemma \ref{lem:degree}, we have $\frac{e_{t+1}(L_{t} \cup M, F_{t+1})}{|L_{t} \cup M|} \geq \frac{r-1}{2} > \frac{r-3}{2}$ and so
	\begin{equation*}
		\frac{e_{t+1}(L_{t} \cup M \cup V(e), F_{t+1})}{|L_{t} \cup M \cup V(e)|}  \leq \frac{e_{t+1}(L_{t} \cup M, F_{t+1})}{|L_{t} \cup M|}.	
	\end{equation*}
	
	The case when $x \in L_{t}$ and $y \notin L_{t}$ is similar. In this case, $e_{t+1}(L_{t} \cup M, F_{t+1})$ counts every edge incident to a vertex in $V(e) \setminus M$ that is neither induced by this set nor connects a vertex in this set to $y$. Since $\{y\} \cup V(e)$ induces a clique in $F_{t+1}$, we have $e_{t+1}(L_{t} \cup M \cup V(e), F_{t+1}) - e_{t+1}(L_{t} \cup M, F_{t+1}) = \binom{m+1}{2}$, and thus
	\begin{align*}
		\frac{e_{t+1}(L_{t} \cup M \cup V(e), F_{t+1})}{|L_{t} \cup M \cup V(e)|}
		& = \frac{e_{t+1}(L_{t} \cup M, F_{t+1}) + \binom{m+1}{2}}{|L_{t} \cup M| + m}\\
		&=\frac{e_{t+1}(L_{t} \cup M, F_{t+1})}{|L_{t} \cup M|} \cdot \frac{|L_{t} \cup M|}{|L_{t} \cup M|+m} + \frac{m+1}{2} \frac{m}{|L_{t} \cup M| + m}\\
		&\leq \frac{e_{t+1}(L_{t} \cup M, F_{t+1})}{|L_{t} \cup M|} \cdot \frac{|L_{t} \cup M|}{|L_{t} \cup M|+m} + \frac{r-1}{2} \frac{m}{|L_{t} \cup M| + m}\\
		& \leq  \frac{e_{t+1}(L_{t} \cup M, F_{t+1})}{|L_{t} \cup M|}.	\qquad \text{(again, by Lemma \ref{lem:degree})}
	\end{align*}
	
	The case $x \notin L_{t}$ and $y \in L_{t}$ is analogous and this completes the proof.
\end{proof}

On the other hand, when the edge $e$ does not satisfy the conditions in Lemma \ref{claim:full-blobs}, the following lemma holds.

\begin{lemma}\label{claim:empty-blobs}
	Let $t \geq 1$, $L_{t} \subseteq V(F_{t}) \setminus \{1,2\}$ with $L_t \neq \emptyset$, $e = \{x,y\} \in E(F_t)$ such that $x,y \notin L_t$, and $ M \subseteq V(F_{t+1}) \setminus V(F_t) $.  If $|V(e) \cap M| > 0$ then
	\[
	\frac{e_{t+1}(L_{t} \cup M, F_{t+1})}{|L_{t} \cup M|} > \min \left \{ \frac{e_{t+1}(L, F_{t+1})}{|L|} : L \subsetneq V(F_{t+1}) \setminus \{1,2\} \right \}\text{.}
	\]
\end{lemma}

\begin{proof}
Set $m = |V(e) \cap M| > 0$ and recall that we have $m \leq r-2$. Since $L_t \neq \emptyset$, we also have $L_{t} \cup M \setminus V(e) \neq \emptyset$. The aim is, by considering the set $L_t \cup M \setminus V(e)$, to show that the set $L_t \cup M$ does not attain the minimum in equation \eqref{eq:eps}.

The edges incident to $V(e) \cap M$ in $F_{t+1}$ can be divided into three groups: those induced by $V(e) \cap M$, those connecting a vertex in $V(e) \cap M$ to a vertex in $V(e) \setminus M$, and those connecting a vertex in $V(e) \cap M$ to either $x$ or $y$. In total, we have
\[
\binom{m}{2} + m(r-2-m) + 2m = \binom{m}{2} + m(r-m)
\]
edges incident to $V(e) \cap M$ in $F_{t+1}$. All these edges are clearly counted by $e_{t+1}(L_{t} \cup M, F_{t+1})$, but not by $e_{t+1}(L_{t} \cup M \setminus V(e), F_{t+1})$ because $x, y \notin L_t$ and $(L_{t} \cup M \setminus V(e)) \cap (V(e) \cap M) = \emptyset$. Thus,
	\begin{align*}
		\frac{e_{t+1}(L_{t} \cup M, F_{t+1})}{|L_{t} \cup M|}
		& = \frac{e_{t+1}(L_{t} \cup M \setminus V(e), F_{t+1}) + \binom{m}{2} + m(r-m)}{|L_{t} \cup M \setminus V(e)| + m}\\
		& = \frac{e_{t+1}(L_{t} \cup M \setminus V(e), F_{t+1})}{|L_{t} \cup M \setminus V(e)|} \cdot \frac{|L_{t} \cup M \setminus V(e)|}{|L_{t} \cup M \setminus V(e)|+m} \\
		& \qquad + \frac{\binom{m}{2} + m(r-m)}{m} \frac{m}{|L_{t} \cup M \setminus V(e)| + m}\\
		& = \frac{e_{t+1}(L_{t} \cup M \setminus V(e), F_{t+1})}{|L_{t} \cup M \setminus V(e)|} \cdot \frac{|L_{t} \cup M \setminus V(e)|}{|L_{t} \cup M \setminus V(e)|+m} \\
		& \qquad + \left ( r-\frac{m+1}{2} \right ) \frac{m}{|L_{t} \cup M \setminus V(e)| + m} \\
		& \geq \frac{e_{t+1}(L_{t} \cup M \setminus V(e), F_{t+1})}{|L_{t} \cup M \setminus V(e)|} \cdot \frac{|L_{t} \cup M \setminus V(e)|}{|L_{t} \cup M \setminus V(e)|+m} \\
		& \qquad + \frac{r+1}{2} \frac{m}{|L_{t} \cup M \setminus V(e)| + m}.
	\end{align*}
	
	If $\frac{e_{t+1}(L_{t} \cup M, F_{t+1})}{|L_{t} \cup M|} \geq \frac{r+1}{2}$ then the claim holds immediately by Lemma \ref{lem:minRatioUpper}. Otherwise we have that
 \[
\frac{e_{t+1}(L_{t} \cup M, F_{t+1})}{|L_{t} \cup M|} > \frac{e_{t+1}(L_{t} \cup M \setminus V(e), F_{t+1})}{|L_{t} \cup M \setminus V(e)|}.
 \]
 This completes the proof.
\end{proof}

Recall that by Lemma \ref{lem:L_tEmpty}, any set $L_{t} \cup M$ that minimises the ratio $\frac{e_{t+1}(L_{t} \cup M, F_{t+1})}{|L_{t} \cup M|}$ has $L_t \neq \emptyset$. Furthermore, Lemma \ref{claim:empty-blobs} tells us that any set of the form $L_{t} \cup M$ with $L_t \neq \emptyset$ that minimises the ratio $\frac{e_{t+1}(L_{t} \cup M, F_{t+1})}{|L_{t} \cup M|}$ has $|V(e) \cap M| = 0$ for every edge $e \in E_t(L_t^c, L_t^c)$. Let us fix $L_t \subseteq V(F_t) \setminus \{1,2\}$ and take $M$ to be maximal such that $L_{t} \cup M$ minimises the edge density.

Assume first that $L_t \subsetneq V(F_t) \setminus \{1,2\}$ with $L_t \neq \emptyset$. By Lemma \ref{claim:full-blobs} we see that we then have $|M| = (r-2)e_t(L_t, F_t)$. Also, let $e$ be an arbitrary edge incident to a vertex $x \in L_t$ in $F_{t+1}$. Again by Lemma \ref{claim:full-blobs} and by the maximality of $M$ we see that $e$ connects $x$ to a vertex in $M$: since we have $x \in L_t$, for every edge $f$ adjacent to $x$ in $F_t$ we set $V(f) \subseteq M$, and then $e$ connects $x$ to a vertex in some $V(f)$. Consequently, all edges incident to $L_{t} \cup M$ in $F_{t+1}$ are incident to $M$, and since $M$ is a union of $e_t(L_t, F_t)$ disjoint cliques on $r-2$ vertices each, such that every vertex in $M$ has exactly two neighbours outside $M$, we have
\begin{equation} \label{eq:smallL_t}
 \frac{e_{t+1}(L_{t} \cup M, F_{t+1})}{|L_{t} \cup M|} = \frac{e_t(L_t, F_t) \left ( \binom{r-2}{2}+2(r-2) \right )}{|L_t|+(r-2)e_t(L_t, F_t)} = \frac{e_t(L_t, F_t) \left ( \binom{r}{2}-1 \right )}{|L_t|+(r-2)e_t(L_t, F_t)}.
\end{equation}

Note that when $L_t = V(F_t) \setminus \{1,2\}$ this choice of $M$ would result in having $ L_{t} \cup M = V(F_{t+1}) \setminus \{1,2\}$, i.e., the edge density with $L_t = V(F_t) \setminus \{1, 2\}$ is minimised among all choices of $M$ by taking the whole graph. As we want to minimise among all possible proper subsets of $V(F_{t+1})$, the case $L_t = V(F_t) \setminus \{1,2\}$ requires some further consideration.

We first consider edges $e \in E(F_t)$ that are incident to a vertex in $\{1,2\}$. For these edges, in the following lemma we show that if, in $F_{t+1}$, at least one vertex of $V(e)$ is in $M$, then the edge density does not go up if we move all but one vertex of $V(e)$ to $M$.

\begin{lemma}\label{claim:all-but-one-blobs-1}
 Let $t \geq 1$, $L_{t} = V(F_t) \setminus \{1,2\}$ and $ M \subsetneq V(F_{t+1}) \setminus V(F_t) $. Let $e = \{x,y\} \in E(F_{t})$ be such that $x \in L_{t}$, $y \in \{1,2\}$, and let $1 \leq |V(e) \cap M| \leq r-3$ with $w \in V(e) \setminus M$. We have
 \[
  \frac{e_{t+1}(L_{t} \cup M, F_{t+1})}{|L_{t} \cup M|} \geq \frac{e_{t+1}(L_{t} \cup M \cup (V(e)\setminus \{w\}), F_{t+1})}{|L_{t} \cup M \cup V(e) \setminus\{w\}|}.
 \]
\end{lemma}

\begin{proof}
We prove this lemma analogously to the second part of the proof of Lemma \ref{claim:full-blobs}. Set $m = |V(e) \setminus M| \leq r-3$. 

We have $x \in L_{t}$ and $y \notin L_{t}$. In this case, $e_{t+1}(L_{t} \cup M, F_{t+1})$ counts every edge adjacent to a vertex in $V(e) \setminus (M \cup \{w\})$ except for those that are induced by this set, and those that connect a vertex in this set to either $y$ or $w$. Since there are $\binom{m-1}{2}+2(m-1)$ such edges not counted by $e_{t+1}(L_t \cup M, F_{t+1})$, we have
\begin{align*}
	\frac{e_{t+1}(L_{t} \cup M \cup (V(e)\setminus \{w\}), F_{t+1})}{|L_{t} \cup M \cup V(e) \setminus\{w\}|}
	& = \frac{e_{t+1}(L_{t} \cup M, F_{t+1})+\binom{m-1}{2}+2(m-1)}{|L_{t} \cup M|+m-1} \\
	& = \frac{e_{t+1}(L_{t} \cup M, F_{t+1})}{|L_{t} \cup M|} \cdot \frac{|L_{t} \cup M|}{|L_{t} \cup M|+m-1} \\
	& \qquad + \frac{m+2}{2} \frac{m-1}{|L_{t} \cup M| + m-1}.
\end{align*}
Since we have $m \leq r-3$, we further obtain
\begin{align*}
	\frac{e_{t+1}(L_{t} \cup M \cup (V(e)\setminus \{w\}), F_{t+1})}{|L_{t} \cup M \cup V(e) \setminus\{w\}|}
	& \leq \frac{e_{t+1}(L_{t} \cup M, F_{t+1})}{|L_{t} \cup M|} \cdot \frac{|L_{t} \cup M|}{|L_{t} \cup M|+m-1} \\
	& \qquad + \frac{r-1}{2} \frac{m-1}{|L_{t} \cup M| + m-1} \\
	& \leq \frac{e_{t+1}(L_{t} \cup M, F_{t+1})}{|L_{t} \cup M|} \text{.}	\qquad \text{(by Lemma \ref{lem:degree})}
\end{align*}
\end{proof}

We now turn our attention to edges $e \in E(F_t)$ that are induced by vertices in $L_t$. Let us show that for them, the edge density does not go up if we move all but one vertex of $V(e)$ to $M$, even if we initially have $V(e) \cap M = \emptyset$.

\begin{lemma}\label{claim:all-but-one-blobs-2}
 Let $t \geq 1$, $L_{t} = V(F_t) \setminus \{1,2\}$ with $ M \subsetneq V(F_{t+1}) \setminus V(F_t) $. Let $e = \{x,y\} \in E(F_{t})$ be such that $x ,y \in L_{t}$, and let $|V(e) \cap M| \leq r-3$ with $w \in V(e) \setminus M$. We have
 \[
  \frac{e_{t+1}(L_{t} \cup M, F_{t+1})}{|L_{t} \cup M|} \geq \frac{e_{t+1}(L_{t} \cup M \cup (V(e)\setminus \{w\}), F_{t+1})}{|L_{t} \cup M \cup V(e) \setminus\{w\}|}.
 \]
\end{lemma}

\begin{proof}
The proof of this lemma is analogous to the first part of the proof of Lemma \ref{claim:full-blobs}. Set $m = |V(e) \setminus M| \leq r-2$. 

We have $x, y \in L_{t}$. In this case, $e_{t+1}(L_{t} \cup M, F_{t+1})$ counts every edge incident to a vertex in $V(e) \setminus (M \cup \{w\})$ except for those that are induced by this set or else connect a vertex in this set to $w$. Since there are $\binom{m-1}{2}+m-1$ such edges, we have
\begin{align*}
	\frac{e_{t+1}(L_{t} \cup M \cup (V(e)\setminus \{w\}), F_{t+1})}{|L_{t} \cup M \cup V(e) \setminus\{w\}|}
	& = \frac{e_{t+1}(L_{t} \cup M, F_{t+1})+\binom{m-1}{2}+m-1}{|L_{t} \cup M|+m-1} \\
	& = \frac{e_{t+1}(L_{t} \cup M, F_{t+1})}{|L_{t} \cup M|} \cdot \frac{|L_{t} \cup M|}{|L_{t} \cup M|+m-1} \\
	& \qquad + \frac{m}{2} \frac{m-1}{|L_{t} \cup M| + m-1} \\
	& \leq \frac{e_{t+1}(L_{t} \cup M, F_{t+1})}{|L_{t} \cup M|} \cdot \frac{|L_{t} \cup M|}{|L_{t} \cup M|+m-1} \\
	& \qquad + \frac{r-2}{2} \frac{m-1}{|L_{t} \cup M| + m-1} \\
	& \leq \frac{e_{t+1}(L_{t} \cup M, F_{t+1})}{|L_{t} \cup M|} \text{.}	\qquad \text{(by Lemma \ref{lem:degree})}
\end{align*}
\end{proof}

Consider now the implications of Lemmas \ref{claim:all-but-one-blobs-1} and \ref{claim:all-but-one-blobs-2} for the case that $L_t = V(F_t) \setminus \{1,2\}$ (note that in this case for all $e \in E(F_t)$ we have $e \in E_t(L_t, F_t)$). Let us again take $M$ to be maximal such that $L_{t} \cup M$ minimises the ratio $\frac{e_{t+1}(L_{t} \cup M, F_{t+1})}{|L_{t} \cup M|}$.

Since $M$ is maximal, we use Lemma \ref{claim:full-blobs} to show that there is exactly one edge $e \in E(F_t)$ such that $|V(e) \cap M| < r-2$. We then use Lemmas \ref{claim:all-but-one-blobs-1} and \ref{claim:all-but-one-blobs-2}. If $e$ is incident to the vertex $1$ (or the vertex $2$), then we have either $|V(e) \cap M| = 0$ or $|V(e) \cap M| = r-3$. If it is not incident to neither $1$ nor $2$ then we have $|V(e) \cap M| = r-3$.

In the former case, the vertices in $V(e) \cup \{1\}$ (or $V(e) \cup \{2\}$, respectively) induce a clique on $r-1$ vertices. The edges of this clique are the only ones not counted by $e_{t+1}(L_{t} \cup M, F_{t+1})$. Since $V(F_{t+1}) \setminus (L_t \cup M) = \{1,2\} \cup V(e)$, this implies that we have
\begin{equation} \label{eq:hugeL_t-1}
 \frac{e_{t+1}(L_{t} \cup M, F_{t+1})}{L_{t} \cup M} \geq \frac{e_{t+1}-\binom{r-1}{2}}{v_{t+1} - r}.
\end{equation}

In the latter case, if $|V(e) \cap M| = r-3$, let $\{w\} = V(e) \setminus M$. Since $V(F_{t+1}) \setminus (L_t \cup M) = \{1,2,w\}$, $w$ can have at most one neighbour, i.e., either vertex $1$ or $2$, not in $L_t \cup M$. This implies that we have
\begin{equation} \label{eq:hugeL_t-2}
 \frac{e_{t+1}(L_{t} \cup M, F_{t+1})}{L_{t} \cup M} \geq \frac{e_{t+1}-1}{v_{t+1} - 3}.
\end{equation}

We are now ready to prove Theorem \ref{thm:epsilon_tLower}.

\begin{proof}[\textit{Proof of Theorem \ref{thm:epsilon_tLower}}]
	The proof proceeds by induction on $t$. Recall the definition of $\varepsilon_t$ given in equation \eqref{eq:eps}.  Our induction hypothesis is the inequality
	\[
	 \varepsilon_{t} \geq \frac{1}{r+1} \left ( \frac{2}{r^2-2} \right )^{t-1}
	\]
 (see inequality \eqref{eq:epsilon_t}). We have already seen in the proof of Lemma \ref{lem:epsilon_tUpper} that $\varepsilon_1 = \frac{1}{r+1}$ so the claim holds for $t=1$. Thus assume the statement is true for some $ t \geq 1$. We now proceed with establishing a recursive lower bound on $e_{t+1}(L_{t+1}, F_{t+1})/|L_{t+1}|$. As before, we consider sets of the form $L_{t+1} = L_{t} \cup M $, where $ L_{t} \subseteq V(F_{t}) \setminus \{1,2\} $ and $ M \subseteq V(F_{t+1}) \setminus V(F_{t}) $, and we write $\ell = |L_{t}|$.  Having established a lower bound on $e_{t+1}(L_{t+1}, F_{t+1})/|L_{t+1}|$ we use \eqref{eq:eps} to bound $\varepsilon_t$ from below. We analyse three cases, depending on the structure of $L_t$ and $M$, and we show that the first of these cases gives the weakest bound on $\varepsilon_t$ from which the theorem follows.
	
	Thus, if there exists a set $L_{t} \cup M $ that minimises the ratio $\frac{e_{t+1}(L_t \cup M, F_{t+1})}{|L_t \cup M|}$ for which we have $L_t \neq V(F_t) \setminus \{1,2\}$ then by \eqref{eq:smallL_t}, we have
	\begin{equation} \label{eq:epsilon_tLower_firstBound}
	 \frac{e_{t+1}(L_t \cup M, F_{t+1})}{\ell + |M|} \geq \frac{\left(\binom{r}{2}-1 \right) e_t(L_t, F_t)}{\ell + (r-2) e_t(L_t, F_t)}= \frac{\frac{(r-2)(r+1)}{2} e_t(L_t, F_t)}{\ell + (r-2) e_t(L_t, F_t)}.
	\end{equation}
	The right-hand side of inequality \eqref{eq:epsilon_tLower_firstBound} is a function of the type $\frac{ax}{b+cx}$ with $a,b,c > 0$. Hence it is increasing in $x$ and therefore we can use relation \eqref{eq:eps} to bound it from below. Thus we have
	\begin{align*}
		\frac{e_{t+1}(L_t \cup M, F_{t+1})}{\ell + |M|}
		& \geq \frac{\frac{(r-2)(r+1)}{2} \ell (1+\varepsilon_t) \frac{e_t}{v_t-2}}{\ell + (r-2) \ell (1+\varepsilon_t) \frac{e_t}{v_t-2}}\\
		& \geq \frac{\frac{(r-2)(r+1)}{2} \ell \lambda (1+c_t)(1+\varepsilon_t)}{\ell + (r-2)\ell \lambda(1+c_t)(1+\varepsilon_t)} 
			&&\text{(by \eqref{eq:crucialRatio})}\\
		& = \lambda (1+c_{t+1}) \frac{\frac{(r-2)(r+1)}{2}\frac{(1+c_t)}{(1+c_{t+1})} (1+\varepsilon_t)}{1+ (r-2)\lambda (1+c_t)(1+\varepsilon_t)}.
	\end{align*} 

	Hence, by \eqref{eq:eps} we see that $\varepsilon_{t+1}$ can be bounded from below by
	\begin{align}
		\varepsilon_{t+1} & \geq \frac{\frac{(r-2)(r+1)}{2}\frac{(1+c_t)}{(1+c_{t+1})} (1+\varepsilon_t)}{1+ (r-2)\lambda (1+c_t)(1+\varepsilon_t)} - 1 \notag\\
		&=\frac{(1+c_t)(1+\varepsilon_t)\left(\frac{(r-2)(r+1)}{2}\frac{1}{1+c_{t+1}} - (r-2)\lambda \right) - 1}{1+(r-2)\lambda(1+c_t)(1+\varepsilon_t)} \notag\\
		&=\frac{(1+c_t)(1+\varepsilon_t)\left(\frac{\tau}{1+c_{t+1}} -\tau+1 \right)-1}{1 + (r-2)\lambda(1+c_t)(1+\varepsilon_t)} && \text{(by \eqref{eq:tau} and \eqref{eq:tauAndLambda})}\notag\\
		&=\frac{(1+c_t)(1+\varepsilon_t)\left(1 - \frac{1}{\tau^t} \right)-1}{1 + (r-2)\lambda(1+c_t)(1+\varepsilon_t)} && \text{(by \eqref{eq:c_t})}\notag\\
		&=\frac{(1+c_t)(1+\varepsilon_t)\frac{1}{1+c_t}-1}{1 + (r-2)\lambda(1+c_t)(1+\varepsilon_t)} && \text{(by \eqref{eq:c_t})}\notag\\
		&=\frac{\varepsilon_t}{1 + (r-2)\lambda(1+c_t)(1+\varepsilon_t)}. \label{eq:eps-t+1-pos1}
	\end{align}

	Using the bound $\varepsilon_t \leq 1/(r+1)$ in Lemma \ref{lem:epsilon_tUpper}, we can bound the expression in \eqref{eq:eps-t+1-pos1} from below by
	\begin{align*}
		\varepsilon_{t+1} & \geq \frac{\varepsilon_t}{1 + (r-2)\lambda(1+c_t)\frac{r+2}{r+1}} \\
		& = \frac{\varepsilon_t}{1 + (r-2)\lambda\frac{\tau^t}{\tau^t-1}\frac{r+2}{r+1}} && \text{(by \eqref{eq:c_t})} \\
		& = \frac{\varepsilon_t}{1 + (\tau - 1)\frac{\tau^t}{\tau^t-1}\frac{r+2}{r+1}}  && \text{(by \eqref{eq:tauAndLambda})}\\
		& \geq \frac{\varepsilon_t}{1 + \tau \frac{r+2}{r+1}} \\
		& = \frac{\varepsilon_t}{1 + \frac{(r+2)(r-2)}{2}}  && \text{(by \eqref{eq:tau})}\\
		& = \frac{2}{r^2-2} \varepsilon_t.
	\end{align*}

	Thus in this case, by the induction hypothesis, we obtain a lower bound on $\varepsilon_{t+1}$ given by
	\begin{align*}
		\varepsilon_{t+1} \geq \frac{2}{r^2-2} \varepsilon_t \geq  \frac{1}{r+1} \left ( \frac{2}{r^2-2} \right )^t.
	\end{align*}
	
	If on the other hand there exists a set $L_t \cup M$ with $L_t = V(F_t) \setminus \{1,2\}$ that gives the minimal ratio in equation \eqref{eq:eps} then we shall consider two cases as outlined in the discussion before this proof. If $M = V(F_{t+1}) \setminus (V(F_t) \cup V(e))$ for some $e \in E(F_t)$ then by \eqref{eq:hugeL_t-1} we have
	\begin{align*}
		\frac{e_{t+1}(L_t \cup M, F_{t+1})}{|L_t \cup M|}
		& \geq \frac{e_{t+1}-\binom{r-1}{2}}{v_{t+1}-r} &&\text{(by \eqref{eq:hugeL_t-1})}\\
		&=\frac{\tau^{t+1}-\binom{r-1}{2}}{(r-2)\frac{(\tau^{t+1}-1)}{\tau-1} - (r-2)} && \text{(by \eqref{eq:e_t} and \eqref{eq:v_t})}\\
		& = \frac{e_{t+1}}{v_{t+1}-2}\left(\frac{1 - \frac{r^2-3r+2}{2\tau^{t+1}}}{1 - \frac{(r-2)(\tau-1)}{(r-2)(\tau^{t+1}-1)}} \right).
	\end{align*}
	As before, by \eqref{eq:eps} we see that in this case $\varepsilon_{t+1}$ is bounded from below by
	\begin{align*}
		\varepsilon_{t+1} & \geq \frac{1 - \frac{r^2-3r+2}{2\tau^{t+1}}}{1 - \frac{\tau-1}{\tau^{t+1}-1}} - 1 \\
		& = \frac{\frac{\tau-1}{\tau^{t+1}-1} - \frac{r^2-3r+2}{2\tau^{t+1}}}{1 - \frac{\tau-1}{\tau^{t+1}-1}} \\
		& = \frac{\tau^{t+2} - \tau^{t+1} - (\tau^{t+1}-1)(\tau-r+2)}{\tau^{t+1}(\tau^{t+1}-1)} \cdot \frac{1}{1 - \frac{\tau-1}{\tau^{t+1}-1}} && \text{(by \eqref{eq:tau})}\\
		& > \frac{(r-3)\tau^{t+1}+(r-2)(\frac{r+1}{2}-1)}{\tau^{t+1}(\tau^{t+1}-1)} && \text{(by \eqref{eq:tau})} \\
		& > \frac{(r-3) (\tau^{t+1}+1)}{\tau^{t+1}(\tau^{t+1}-1)} \\
		& > \frac{r-3}{\tau^{t+1}}.
	\end{align*}
	 The last inequality is clearly weaker than it could be and is given in this form to make further calculations simpler.
	
	If $L_t = V(F_t) \setminus \{1,2\}$ and $M = V(F_{t+1}) \setminus (V(F_t) \cup \{w\})$ for some $w \in V(F_{t+1}) \setminus V(F_t)$ then by \eqref{eq:hugeL_t-2} we have
	\begin{align*}
		\frac{e_{t+1}(L_t \cup M, F_{t+1})}{|L_t \cup M|}
		& \geq \frac{e_{t+1}-1}{v_{t+1}-3} &&\text{(by \eqref{eq:hugeL_t-2})}\\
		&=\frac{\tau^{t+1}-1}{(r-2)\frac{(\tau^{t+1}-1)}{\tau-1} - 1} && \text{(by \eqref{eq:e_t} and \eqref{eq:v_t})}\\
		& = \frac{e_{t+1}}{v_{t+1}-2}\left(\frac{1 - \frac{1}{\tau^{t+1}}}{1 - \frac{\tau-1}{(r-2)(\tau^{t+1}-1)}} \right). && \text{(by \eqref{eq:v_t-2})}
	\end{align*}
	Thus, we see that by \eqref{eq:eps} in this case $\varepsilon_{t+1}$ is at least
	\begin{align*}
		\varepsilon_{t+1} & \geq \frac{1 - \frac{1}{\tau^{t+1}}}{1 - \frac{\tau-1}{(r-2)(\tau^{t+1}-1)}} - 1 \\
		& = \frac{\frac{\tau-1}{(r-2)(\tau^{t+1}-1)} - \frac{1}{\tau^{t+1}}}{1 - \frac{\tau-1}{(r-2)(\tau^{t+1}-1)}} \\
		& = \frac{\tau^{t+2} - (r-1)\tau^{t+1} + (r-2)}{\tau^{t+1}(r-2)(\tau^{t+1}-1)} \cdot \frac{1}{1 - \frac{(\tau-1)}{(r-2)(\tau^{t+1}-1)}} \\
		& > \frac{\tau^{t+2} - (r-1)\tau^{t+1}}{\tau^{t+1}(r-2)(\tau^{t+1}-1)} \\
		& > \frac{\tau - (r-1)}{(r-2) \tau^{t+1}} \\
		& = \frac{1}{\tau^{t+1}} \left ( \frac{r+1}{2} - \frac{r-1}{r-2} \right ) && \text{(by \eqref{eq:tau})} \\
		& = \frac{1}{\tau^{t+1}} \left ( \frac{r+1}{2} - \frac{1}{r-2} - 1 \right ) \\
		& = \frac{\lambda - 1}{\tau^{t+1}} && \text{(by \eqref{eq:lambda})}. 
	\end{align*}
	
	It remains to show that for all $t \geq 0$ we have 
	\begin{equation}\label{eq:eps-est}
	 \frac{1}{r+1} \left ( \frac{2}{r^2-2} \right )^t \leq \min \left \{ \frac{r-3}{\tau^{t+1}}, \frac{\lambda - 1}{\tau^{t+1}} \right \}.
	\end{equation}
	To see that the inequality in \eqref{eq:eps-est} holds, note that since $r \geq 4$,
	\[
	 \frac{1}{r+1} = \frac{r-2}{(r+1)(r-2)} \leq \frac{\frac{r+1}{2} - \frac{1}{r-2} - 1}{\frac{(r+1)(r-2)}{2}} = \frac{\lambda - 1}{\tau} \text{,}
	\]
	and
	\[
	 \frac{1}{r+1} \leq \frac{1}{r+1}\frac{2(r-3)}{r-2} = \frac{r-3}{\tau}.
	\]
	Moreover,
	\[
	 \frac{2}{r^2-2} < \frac{2}{r^2-r-2} = \frac{1}{\tau}.
	\]
	Thus, for any $t \geq 0$, the terms in equation \eqref{eq:eps-est} satisfy $\frac{1}{r+1}\left(\frac{2}{r^2-2} \right)^t \leq \frac{\lambda-1}{\tau} \cdot \frac{1}{\tau^t}$ and $\frac{1}{r+1}\left(\frac{2}{r^2-2} \right)^t \leq \frac{r-3}{\tau} \cdot \frac{1}{\tau^t}$ . This completes the induction on $t$ and the proof of Theorem \ref{thm:epsilon_tLower}.
\end{proof}

Let us conclude this section by commenting on the sharpness of the bound in \eqref{eq:epsilon_t}. We know that $\varepsilon_1 = 1/(r+1)$ is obtained by taking $L = L_1$ of size $r-3$, i.e., by leaving just one vertex in $V(F_1)\setminus \{1,2\}$ outside $L_1$. If we then continue by, for $2 \leq i \leq t$, taking $L_i$ to be the union of $L_{i-1}$ and all vertices in $V(F_i) \setminus V(F_{i-1})$ that are adjacent to at least one vertex in $L_{i-1}$, then the resulting set $L_t$ has $1+2(r-2)\frac{\tau^{t-1}-1}{\tau-1}$ vertices and $\tau^{t-1}(\tau-2)$ edges adjacent to it. This shows that for some $C_r > 0$, one can obtain a bound $\varepsilon_t \leq C_r / \tau^{t}$. Since $\tau = (r^2-r-2)/2$, this implies that our lower bound on $\varepsilon_t$ is relatively sharp.

We have thus shown that the edge density of all proper subsets of $V(F_{t})$ is strictly bounded from below by $ e_{t} / (v_{t} - 2) $. As we will see in Section \ref{sec:upperBound}, we are now equipped with the necessary means to prove Statement \eqref{stat:1} of Theorem \ref{thm:criticalProb}.

\section{Upper bound on the critical probability}
\label{sec:upperBound}

In this section we prove Statement \eqref{stat:1} of Theorem \ref{thm:criticalProb}. We shall use the following form of Janson's inequality.
\begin{theorem}
\label{thm:JansonIneq}
 Let $R$ be a set and let $S \subseteq R$ be a random subset of $R$, where each $r \in R$ is in $S$ independently with probability $p$. Let $\{B_1, \ldots, B_m \}$ be a collection of finite subsets of $R$ and let $C_i$ be the event that $B_i \subseteq S$. Let $Z = \sum_{i=1}^m \indicator{C_i}$ and let $\mu = \sum_{i=1}^m \PP_p (C_i) = \EE [Z]$. For $1 \leq i , j \leq m$, $i \neq j$, let $i \sim j$ if $B_i \cap B_j \neq \emptyset$, i.e., if the events $C_i$ and $C_j$ are dependent, and let $\Delta = \sum_{i \sim j} \PP (C_i \cap C_j)$. Then
\begin{equation}
 \label{eq:JansonIneq}
 \PP_p (Z=0) \leq e^{-\mu + \Delta/2 }.
\end{equation}
\end{theorem}
\qed 

Let $e_0 = \{1,2\}$. In this section we show that if $p(n) \geq n^{-\frac{v_t-2}{e_t}} \log n$ then we have $\PP_{p_n} (e_{0} \notin E(G_t)) \leq n^{-3}$. By the union bound this implies $\PP_{p_n} (T \leq t) \to 1$.

\begin{proofOfStatementOfTheorem}{\eqref{stat:1}}{thm:criticalProb}
Fix $n$, sufficiently large, and $t = t(n) \leq \frac{\log\log n}{3\log \tau}$.

As always, fix two vertices $1$ and $2$ and let $e_0 = \{1,2\}$.  Given any of the $\binom{n-2}{v_t-2}$ subsets $X_i \subseteq \{3,4,\ldots,n\}$ of size $(v_t-2)$, let $F_t (X_i)$ be an arbitrary fixed copy of the graph $F_t$ on $X_i \cup \{1,2\}$ that adds the edge $e_{0}$ to the graph $G$ in $t$ time steps. We shall apply Theorem \ref{thm:JansonIneq} to bound the probability of $e_{0}$ not being added to the graph by time $t$ from above. 

Let $p(n) = n^{-(v_t-2)/e_t} \log n$, as in Statement \eqref{stat:1} and let $G = G_{n,p(n)}$. In Theorem \ref{thm:JansonIneq} we shall take $R = [n]^{(2)}$, $S = E(G)$ and for $i=1, \ldots, \binom{n-2}{v_t-2}$ let $B_i = E(F_t(X_i))$. We define $C_i$, as well as $Z$, as in Theorem \ref{thm:JansonIneq}. We clearly have $\mu = \sum_{i=1}^m \PP_p (C_i) = \binom{n-2}{v_t-2} p^{e_t}$.

An upper bound on $\Delta$ in Theorem \ref{thm:JansonIneq} can be obtained by considering the following form,
\[
\Delta =\sum_{i \sim j} \PP(C_i \cap C_j)  =\sum_{1 \leq |X_i \setminus X_j| \leq v_{t-3}} \PP(B_i \subseteq S \text{ and } B_j \subseteq S).
\]
There are $\binom{n-2}{v_t-2}$ ways to choose the set $X_i$ and, having fixed $X_i$, for $1 \leq k \leq v_t-3$ there are $\binom{v_t-2}{v_t-2-k} \binom{n-v_t}{k}$ ways to choose the set $X_j$ such that $|X_j \setminus X_i| = k$. We also have $|B_i| = e_t$ and, by \eqref{eq:eps}, $|B_j \setminus B_i| \geq k(1+\varepsilon_t)\frac{e_t}{v_t-2}$. Hence

\begin{align}
	\Delta &\leq \binom{n-2}{v_t-2} \sum_{k = 1}^{v_t-3} \binom{v_t-2}{v_t-2-k} \binom{n-v_t}{k} p^{e_t + k (1+\varepsilon_t)\frac{e_t}{v_t-2}} \notag\\
	& = \binom{n-2}{v_t-2} p^{e_t} \sum_{k = 1}^{v_t-3} \binom{v_t-2}{k} \binom{n-v_t}{k} p^{k (1+\varepsilon_t)\frac{e_t}{v_t-2}} \notag\\
	& = \mu \sum_{k=1}^{v_t-3} \binom{v_t-2}{k} \binom{n-v_t}{k} p^{k (1+\varepsilon_t)\frac{e_t}{v_t-2}} \notag\\
	& < \mu \sum_{k=1}^{v_t-3} ((v_t-2) n)^k p^{k (1+\varepsilon_t)\frac{e_t}{v_t-2}}. \label{eq:Delta-bd1}
\end{align}
Note that, by the definition of $\varepsilon_t$ in \eqref{eq:eps}, by Lemma \ref{lem:minRatioUpper} and \eqref{eq:lambda},
\begin{equation} \label{eq:exponentUpperBound}
\frac{e_t}{v_t-2}(1 +\varepsilon_t) \leq \frac{r+1}{2} = \lambda + \frac{1}{r-2} < 2 \lambda.
\end{equation}
By \eqref{eq:tauAndLambda}, for all $r \geq 4$ we have $\tau - 1 > r-2$ hence, by \eqref{eq:v_t-2}, $v_t-2 < \tau^t$. For $t \leq \frac{\log \log n}{3 \log \tau} = \frac{\log_\tau (\log n)}{3}$, by \eqref{eq:tauAndLambda} we have
\begin{equation} \label{eq:v_tUpperBound}
 v_t-2 < \tau^{t} \leq (\log n)^{1/3}.
\end{equation}
Now, using the fact that $p(n) =  n^{-(v_t-2)/e_t} \log n$, we have
\begin{align}
\sum_{k=1}^{v_t-3} ((v_t-2) n)^k p^{k \frac{e_t}{v_t-2}(1+\varepsilon_t)}
	& = \sum_{k=1}^{v_t-3} ((v_t-2) n)^k n^{-k(1+\varepsilon_t)} (\log n)^{k \frac{e_t}{v_t-2}(1+\varepsilon_t)} \notag\\
	& \leq \sum_{k=1}^{v_t-3} ((v_t-2) n)^k (\log n)^{2\lambda k} n^{-k(1+\varepsilon_t)} && \text{(by \eqref{eq:exponentUpperBound})} \notag\\
	& \leq \sum_{k=1}^{v_t - 3} (\log n)^{k/3} (\log n)^{2\lambda k} n^{k-k(1+\varepsilon_t)} && \text{(by \eqref{eq:v_tUpperBound})} \notag\\
	& = \sum_{k=1}^{v_t - 3} (\log n)^{k(2\lambda + 1/3)} n^{-k\varepsilon_t}. \label{eq:Delta-bd2}
\end{align}
By inequality \eqref{eq:epsilon_t} in Theorem \ref{thm:epsilon_tLower}, we have $\varepsilon_t \geq \frac{1}{r+1}(2/(r^2-2))^{t-1}$. This implies that for $t \leq \frac{\log\log n}{3\log \tau}$ we have $\varepsilon_t > \left(\log n\right)^{-1/2}$. Indeed,
\[
\varepsilon_t \geq \frac{1}{r+1} \left ( \frac{2}{r^2-2} \right )^{t-1} > \frac{1}{r+1}\left ( \frac{1}{\tau}\frac{2 \tau}{r^2-2} \right )^{\frac{\log\log n}{3\log \tau}} = \frac{1}{r+1} (\log n)^{-1/3} \left (\frac{r^2-r-2}{r^2-2} \right )^{\frac{\log\log n}{3\log \tau}}.
\]
Now, since the sequence $a_r = \frac{r^2-r-2}{r^2-2}$ is increasing in $r$ and for all $r \geq 4$ we have $\tau \geq 5$,
\[
\left (\frac{r^2-r-2}{r^2-2} \right )^{\frac{1}{3\log \tau}} \geq \left (\frac{5}{7} \right )^{\frac{1}{3\log 5}} > 0.93 > e^{-1/10},
\]
and the bound $\varepsilon_t \geq \left(\log n\right)^{-1/2}$ follows for $n$ large enough. Consequently, for $n$ large enough we have that $(\log n)^{2\lambda + 1/3} n^{-\varepsilon_t} \leq (\log n)^{2\lambda + 1/3} \exp(-(\log n)^{1/2}) < 1/2$. Hence continuing the string of inequalities from \eqref{eq:Delta-bd1} and \eqref{eq:Delta-bd2},
\begin{equation} \label{eq:deltaBound}
 \Delta \leq \mu \sum_{k = 1}^{v_t - 3} \left ( \left (\log n \right )^{2\lambda + 1/3} n^{-\varepsilon_t}\right )^k \leq \mu \sum_{k = 1}^{\infty} \left (\frac{1}{2} \right )^k = \mu.
\end{equation}
By \eqref{eq:v_tUpperBound}, we see that the ratio $(\frac{n-2}{n})^{v_t}$ tends to 1 as $n \to \infty$. Thus, for $p \geq n^{-\frac{v_t-2}{e_t}} \log n$ we have
\begin{equation} \label{eq:muBound}
\mu = \binom{n-2}{v_t-2} p^{e_t} \geq \left ( \frac{n}{v_t-2} \right )^{v_t-2} n^{-\frac{v_t-2}{e_t}e_t} ( \log n )^{e_t} = \frac{(\log n)^{e_t}}{(v_t-2)^{v_t-2}} \geq \left(\frac{\log n}{\tau^t}\right)^{\tau^t},
\end{equation}
where the second inequality follows from \eqref{eq:e_t} and the bound on $v_t$ in \eqref{eq:v_tUpperBound}. Hence using Theorem \ref{thm:JansonIneq} we obtain
\begin{align*}
\PP(Z = 0) & \leq \exp\left(-\mu + \Delta/2\right) && \text{(by \eqref{eq:JansonIneq})} \\
	& \leq \exp(-\mu/2) && \text{(by \eqref{eq:deltaBound})}\\
	& \leq \exp\left(-\frac{1}{2} \left(\frac{\log n}{\tau^t} \right)^{\tau^t} \right). && \text{(by \eqref{eq:muBound})}
\end{align*}

Note that the function $x \mapsto \left(\frac{\log n}{x}\right)^x$ is increasing for $x \in (0,\frac{\log n}{e} ]$.  When $t = 1$, for $n$ sufficiently large,
\[
\left(\frac{\log n}{\tau}\right)^{\tau} = \log n \frac{(\log n)^{\tau -1}}{\tau^\tau} \geq 6 \log n.
\] 

For $t \leq \frac{\log \log n}{3\log \tau}$ we have $\tau \leq \tau^t \leq (\log n)^{1/3} < \frac{\log n}{e}$ and so, 
\[
\left(\frac{\log n}{\tau^t} \right)^{\tau^t}  \geq \left(\frac{\log n}{\tau}\right)^{\tau} \geq 6 \log n
\]
when $n$ is sufficiently large. Thus,
\[
\PP(Z = 0) \leq \exp\left(-\frac{1}{2} \left(\frac{\log n}{\tau^t} \right)^{\tau^t} \right) \leq \exp(-3 \log n) = n^{-3},
\]
and applying the union bound yields
\[
 \PP_{p_n} (T \leq t) \geq 1 - \frac{1}{n}.
\]
This completes the proof of Statement \eqref{stat:1} of Theorem \ref{thm:criticalProb}.
\end{proofOfStatementOfTheorem}

\section{Lower bound on the critical probability}
\label{sec:lowerBound}

In this section we prove Statement \eqref{stat:2} of Theorem \ref{thm:criticalProb}. More precisely, we show that if $p(n) = o \left ( n^{-\frac{v_t-2}{e_t}} \right )$ then even a single fixed pair $e_{0} = \{1,2\}$ is not added to the graph by time $t$ with high probability.

\begin{proofOfStatementOfTheorem}{\eqref{stat:2}}{thm:criticalProb}
Recall that $E(G_t)$ denotes the edges of the graph after $t$ time steps. We have
\[
\PP_{p_n} (T \leq t) \leq \PP_{p_n} (e_{0} \in E(G_t)) \leq \sum_{i=0}^t \PP_{p_n} (e_{0} \text{ is added at time } i).
\]
The probability that $e_0$ is added to the graph at time $0$ is clearly $p_n = o(1)$. For $1 \leq i \leq t-1$ we can be very generous with our estimates. The number of vertices of any minimal graph that adds $e_0$ to the graph at time $i \geq 1$ is at least $r$ (including $1$ and $2$) and, by Lemma \ref{lem:minimalsSmall}, at most $v_i$. The number of different graphs on a set of $j$ vertices is $2^{\binom{j}{2}} < 2^{j^2}$. Finally, by Corollary \ref{cor:manyEdges} any minimal graph on $j$ vertices that adds $e_0$ to the graph contains at least $\lambda (j-2) + 1$ edges. Since we take $p(n) = o \left ( n^{-\frac{v_t-2}{e_t}} \right )$, we see that by the union bound the probability that $e_0$ is added to the graph at some time $1 \leq i \leq t-1$ is at most
\begin{align}
 \sum_{i=1}^{t-1}\sum_{j = r}^{v_i} \binom{n}{j-2} 2^{\binom{j}{2}} p^{\lambda(j-2)+1} 
 	&\leq \sum_{i=1}^{t-1} \sum_{j=r}^{v_i} n^{j-2} 2^{j^2} n^{-(\lambda (j-2) + 1)\frac{v_t-2}{e_t}} \notag\\
	& = \sum_{i=1}^{t-1} \sum_{j=r-2}^{v_i-2} n^j 2^{(j+2)^2} n^{-(\lambda j + 1)\frac{v_t-2}{e_t}}.  \label{eq:tooSoon}
\end{align}
For $i=t$ we further divide into two cases. We deal with the graphs on at most $v_t-1$ vertices in the same way as we did for $i<t$. By union bound, the probability that any one of them appears in the graph is at most
\begin{equation} \label{eq:tooSmall}
 \sum_{j=r-2}^{v_t-3} n^j 2^{(j+2)^2} n^{-(\lambda j + 1)\frac{v_t-2}{e_t}}.
\end{equation}
The last case we need to consider are graphs on $v_t$ vertices that add $e_0$ to the graph at time $t$. By Corollary \ref{cor:largestIsomorphic} there are at most $n^{v_t-2}$ such graphs and the probability that we obtain at least one of them, for some $\omega(n)$ that tends to infinity as $n \to \infty$, is at most
\[
 n^{v_t-2} n^{-e_t \frac{v_t-2}{e_t}} / \omega(n)^{e_t} = \omega(n)^{-e_t} = o(1).
\]
Consequently we have
\begin{equation}
\label{eq:lowerBound}
\PP_{p_n} (T \leq t) \leq \sum_{i=1}^{t-1} \sum_{j=r-2}^{v_i-2} n^j 2^{(j+2)^2} n^{-(\lambda j + 1)\frac{v_t-2}{e_t}} + \sum_{j=r-2}^{v_t-3} n^j 2^{(j+2)^2} n^{-(\lambda j + 1)\frac{v_t-2}{e_t}} + o(1).
\end{equation}
For all $t \geq 1$, by \eqref{eq:v_t-2} we have
\[
 v_{t-1}-2 = (r-2) \frac{\tau^{t-1}-1}{\tau-1} < \frac{r-2}{\tau} \frac{\tau^{t}-1}{\tau-1} = \frac{v_{t}-2}{\tau}.
\]
It follows that for $i \leq t-1$ we can bound the powers of $n$ in \eqref{eq:tooSoon} by
\begin{align*}
 j - (\lambda j + 1)\frac{(v_t-2)}{e_t}
    & = j - j \frac{\lambda (v_t-2)}{e_t} - \frac{v_t-2}{e_t}\\
    & = j \left(1 - \frac{1}{1+c_t} \right) - \frac{1}{\lambda(1+c_t)} && \text{(by \eqref{eq:crucialRatio})}\\
    & = \left(j c_t - \frac{1}{\lambda} \right) \frac{1}{1+c_t}\\
    & \leq \left((v_{t-1}-2) c_t - \frac{1}{\lambda} \right)\frac{1}{1+c_t}\\
    & = \left(\frac{r-2}{\tau-1}(\tau^{t-1} - 1) \frac{1}{\tau^t-1} - \frac{r-2}{\tau-1} \right)\frac{1}{1+ \frac{1}{\tau^t-1}} && \text{(by \eqref{eq:v_t-2}, \eqref{eq:c_t} and \eqref{eq:tauAndLambda})}\\
    & = \frac{r-2}{\tau-1} \frac{\tau^{t-1}-\tau^t}{\tau^t-1}  \frac{\tau^t-1}{\tau^t} \\
    & = - \frac{r-2}{\tau}.
\end{align*}
Analogously, for $i = t$ and $j \leq v_t-3$, we can bound the powers of $n$ in \eqref{eq:tooSmall} by
\begin{align*}
 j - (\lambda j + 1) \frac{(v_t-2)}{e_t}
  &=\left(j c_t - \frac{1}{\lambda} \right)\frac{1}{1+c_t}\\
  &\leq \left((v_t - 3) c_t - \frac{1}{\lambda} \right)\frac{1}{1+c_t}\\
  &=\left(\frac{\tau^t-1}{\lambda} \frac{1}{\tau^t-1} - c_t - \frac{1}{\lambda} \right) \frac{1}{1+c_t} && \text{(by \eqref{eq:v_t-2} and \eqref{eq:c_t})}\\
  &=\frac{-c_t}{1+c_t} = \frac{-1/(\tau^t-1)}{1+1/(\tau^t-1)}\\
  &=\frac{-1}{\tau^t}.
\end{align*}
We can use these estimates and the fact that, by \eqref{eq:v_t-2}, for any $i \geq 1$ we have $v_i - 2  = \frac{\tau^i - 1}{\lambda} < \tau^i/\lambda$, to bound $\PP_{p_n} (T \leq t)$ from above. Indeed,
\begin{align}
\label{eq:P_nUpper}
 \PP_{p_n} (T \leq t) & \leq \sum_{i=1}^{t-1} \sum_{j=r-2}^{v_i-2} 2^{(j+2)^2} n^{-(r-2)/\tau} + \sum_{j=r-2}^{v_t-3} 2^{(j+2)^2} n^{-1/\tau^t} + o(1) \notag \\
  & \leq (t-1)(v_{t-1}-r+1) 2^{(v_{t-1})^2} n^{-(r-2)/\tau} + (v_t-r) 2^{(v_t)^2} n^{-1/\tau^t} + o(1) \notag \\
  & < t\frac{\tau^{t-1}}{\lambda}2^{\tau^{2(t-1)}/\lambda^2} n^{-(r-2)/\tau} + \frac{\tau^t}{\lambda}2^{\tau^{2t}/\lambda^2} n^{-1/\tau^t} + o(1).
\end{align}
There is some constant $C_r' > 0$ such that for all $t \geq C_r'$ we have
\[
 2^{\tau^{2(t-1)}/\lambda^2} \geq t \frac{\tau^{t-1}}{\lambda} \qquad \text{and} \qquad 2^{\tau^{2t}(\lambda^2-1)/\lambda^2} \geq \frac{\tau^t}{\lambda}.
\]
For $t < C_r'$ all three terms in \eqref{eq:P_nUpper} tend to $0$ as $n \to \infty$ and we clearly have $\PP_{p_n}(T \leq t) = o(1)$. For $t \geq C_r'$ we continue \eqref{eq:P_nUpper} to obtain
\begin{align*}
 \mathbb{P}_{p_n}(T \leq t)
  & \leq 2^{2\tau^{2(t-1)}/\lambda^2} n^{-(r-2)/\tau} + 2^{\tau^{2t}} n^{-1/\tau^t} + o(1) \\
  & \leq \exp\left(\frac{2\tau^{2t}}{\lambda^2} \log 2 - \frac{(r-2)}{\tau}\log n \right) + \exp\left(\tau^{2t}\log 2 - \frac{1}{\tau^t}\log n \right) + o(1).
\end{align*}
Thus, for $C_r' \leq t \leq \frac{\log \log n}{3\log \tau}$, with all logarithms having base $e$,
\begin{align*}
\mathbb{P}_{p_n}(T \leq t)
	&\leq \exp\left(\frac{2(\log n)^{2/3}}{\lambda^2}\log 2 - \frac{(r-2)}{\tau}\log n\right)\\
	&\qquad + \exp\left((\log n)^{2/3} \log 2  -(\log n)^{2/3} \right) + o(1) \\
	&= o(1).
\end{align*}

This completes the proof of Theorem \ref{thm:criticalProb}.
\end{proofOfStatementOfTheorem}

\section{Open problems}
\label{sec:questions}

In this paper we determine the critical probability for percolation by time $t$ in $K_r$-bootstrap percolation up to a logarithmic factor. The first obvious problem to consider is the following.
\begin{problem}
 Close the gap between Statement \eqref{stat:1} and Statement \eqref{stat:2} in Theorem \ref{thm:criticalProb}.
\end{problem}
We do not expect neither of our bounds to be sharp. However, we believe that for the range of $t$ discussed in this paper we have  $p_c(n, r, t) / n^{-(v_t-2)/e_t} \to \infty$.

The second open problem we pose here is of extremal nature. Lemma \ref{lem:minimalsSmall} tells us that minimal graphs adding $e_0$ to the graph at time $t$ have at most $(r-2) \frac{\tau^t-1}{\tau-1}+2$ vertices and $\tau^t$ edges.
\begin{problem}
 How small, both in terms of the size of the vertex set and the edge set, can minimal graphs adding $e_0$ at time $t$ be?
\end{problem}

\section*{Acknowledgement}

The authors wish to thank the referees for their very detailed feedback and, in particular, for their many suggestions that helped to improve the presentation of a number of the technical lemmas.

\bibliographystyle{amsplain}

 \bibliography{time-paper-arxiv}

\end {document}